\documentclass{amsart}
\usepackage[enumitem,theorem,mathbb,hyperref]{paper_diening}

\usepackage[backend=biber,maxbibnames=5,maxalphanames=5,style=alphabetic,bibencoding=utf8,giveninits,url=false,isbn=false]{biblatex}
\AtBeginBibliography{\small}

\usepackage{a4wide}
\usepackage{microtype}
\usepackage{tikz}
\usetikzlibrary{intersections,calc,arrows.meta,patterns,angles,quotes}
\usetikzlibrary{backgrounds}

\providecommand{\logmean}[1]{\mean{#1}^{\log}}

\providecommand{\Rnnpos}{{\setR^{n \times n}_{>0}}}

\renewcommand{\rho}{\varrho}

\title[Global Maximal Regularity for equations with degenerate weights]{ Global Maximal Regularity for equations with degenerate weights}

\author[Balci]{Anna Kh.~Balci}
\address{Fakult\"at für Mathematik, University Bielefeld,
Universit\"atsstrasse 25, 33615 Bielefeld, Germany}
\email{akhripun@math.uni-bielefeld.de}
\author[Byun]{Sun-Sig Byun}
\address{Department of Mathematical Sciences and Research Institute of Mathematics, Seoul National University, 1, Gwanak-ro, Gwanak-gu, Seoul, Republic of Korea}
\email{byun@snu.ac.kr}
\author[Diening]{Lars Diening}
\address{Fakult\"at für Mathematik, University Bielefeld,
Universit\"atsstrasse 25, 33615 Bielefeld, Germany}
\email{lars.diening@uni-bielefeld.de}
\author[Lee]{Ho-Sik Lee}
\address{Department of Mathematical Sciences, Seoul National University, 1, Gwanak-ro, Gwanak-gu, Seoul, Republic of Korea }
\email{lshnsu92@snu.ac.kr}

\begin{document}

\begin{abstract}
  In this paper we are concerned with global maximal regularity estimates for elliptic equations with degenerate weights. We consider both the linear case and the non-linear case.  We show that higher integrability of the gradients can be obtained by imposing a local small oscillation condition on the weight and a local small Lipschitz condition on the boundary of the domain.  Our results are new  in the linear and non-linear case. We show by example that the relation between the exponent of higher integrability and the smallness parameters is sharp even in the linear or the unweighted case.
\end{abstract}

\keywords{degenerate weight; Lipschitz boundary; $p$-Laplacian; regularity of solutions}

\thanks{Anna Kh.Balci and Lars Diening thank for the financial support the Deutsche Forschungsgemeinschaft (DFG, German Science Foundation) through SFB 1283/2 2021-317210226 at Bielefeld University.\\Sun-Sig Byun thanks for the financial support from the National Research Foundation of Korea (NRF) through NRF-2021R1A4A1027378 at Seoul National University, Ho-Sik Lee thanks for the financial support from the National Research Foundation of Korea (NRF) through IRTG 2235/NRF-2016K2A9A2A13003815 at Seoul National University.
}

\subjclass[2010]{%
35J92, 
35J70, 
35J25, 
35B65, 
46E35 
}

\maketitle

\numberwithin{theorem}{section}
\numberwithin{equation}{section}


\section{Introduction}
\label{sec:introduction}

We study the following degenerate elliptic equation of the form
\begin{align}\label{eq:lap}
  \begin{alignedat}{3}
    -\divergence(\mathbb{A}(x)\nabla u) &=-\divergence (\mathbb{A}(x)F) &\qquad &\text{in }\Omega,
    \\
    u &=0&\qquad&\text{on }\partial\Omega,
  \end{alignedat}
\end{align}
in the linear case, and of the form
\begin{align}
  \label{eq:plap}
  \begin{alignedat}{3}
    -\divergence(|\setM(x)\nabla u|^{p-2}\setM^2(x)\nabla u) &=-\divergence (|\setM(x)F|^{p-2}\setM^2(x)F) &\qquad &\text{in }\Omega,
    \\
    u &=0 &\qquad&\text{on }\partial\Omega,
  \end{alignedat}
\end{align}
in the non-linear case. We often write $\bbM(x)$ to emphasize the dependence of the weight on~$x$. Here, $\Omega\subset\setR^n$ is a bounded domain with $n\geq 2$, $1<p<\infty$, $F:\Omega\rightarrow\setR^n$ is a given vector-valued function, $\mathbb{M}:\setR^n\rightarrow\setR^{n\times n}$ is a given symmetric and positive definite matrix-valued weight satisfying
\begin{align}
  \label{eq:M}
  \abs{\setM(x)}\, \abs{\setM^{-1}(x)}\leq\Lambda\quad(x\in\setR^n)
\end{align} 
for some constant $\Lambda\geq 1$, where~$\abs{\cdot}$ is the spectral norm, and $\mathbb{A}(x):=\setM^2(x)$.  This condition says that $\bbM$ has a uniformly bounded condition number.  Note that a right-hand side of the form $-\divergence G$ with $G\,:\, \Omega \to \setR^n$ can be immediately rewritten in the above form in terms of~$F$.  Note that~\eqref{eq:lap} is a special case of \eqref{eq:plap} for~$p=2$. The condition~\eqref{eq:M} in this case reads as
\begin{align}
  \label{eq:A}
  \abs{\bbA(x)}\, \abs{\bbA^{-1}(x)} \leq\Lambda^2\quad(x\in\setR^n).
\end{align}
Let us define the scalar weight
\begin{align}
  \label{eq:omega}
  \omega(x)=|\setM(x)|=\sqrt{\abs{\bbA(x)}}.
\end{align}

Suppose that $\omega^p$ is an $\mathcal{A}_p$-Muckenhoupt weight (see Section 2) and $F\in L^{p}_{\omega}(\Omega):=L^p(\Omega,\omega^p\,dx)$, then there exists a unique weak solution $u\in W^{1,p}_{0,\omega}(\Omega)$ of \eqref{eq:plap}, which means
\begin{align}\label{eq:1st}
	\int_{\Omega}|\setM(x)\nabla u|^{p-2}\setM^2(x)\nabla u\cdot\nabla \phi\,dx=\int_{\Omega}|\setM(x)F|^{p-2}\setM^2(x)F\cdot\nabla \phi\,dx
\end{align}
for all $\phi\in W^{1,p}_{0,\omega}(\Omega)$. Moreover,  we have the following standard energy estimate
\begin{align}\label{eq:energy}
	\int_{\Omega}|\nabla u|^p\omega^p\,dx\leq c\int_{\Omega}|F|^p\omega^p\,dx
\end{align}
with $c=c(n,p,\Lambda)$, see \cite{BDGN,CMP1}.

This paper aims for proving the following global maximal regularity estimates
\begin{align}\label{eq:CZ}
  \int_{\Omega}|\nabla u|^q\omega^q\,dx\leq c\int_{\Omega}|F|^q\omega^q\,dx
\end{align}
for every $q\in(1,\infty)$ in the linear case \eqref{eq:lap}, and for every $q\in[p,\infty)$ in the non-linear case \eqref{eq:plap}. The positive constant $c$ is independent of $F$ and $u$, under minimal extra  assumptions on both the boundary of $\Omega$ and the weight $\setM$ in addition to \eqref{eq:M}. We pay special attention to the optimal dependence of the parameters of the boundary and of the coefficients on $q$. Estimates of this type are also known under the name of global \emph{non-linear Calder\'{o}n-Zygmund estimates}.  Our main results are presented in Theorem~\ref{thm:lin} and~\ref{thm:nonlin}.

We first explain the linear case. When $\mathbb{A}$ is the identity matrix, our result is related to \cite{CZ1, CZ2}, from which the linear Calder\'{o}n-Zygmund theory originates. If $\mathbb{A}(x)$ is assumed to be only measurable, but uniformly elliptic in the sense that
\begin{align}\label{eq:unif.ellip}
	\lambda_{\min}|\xi|^2\leq\left<\mathbb{A}(x)\xi,\xi\right>\leq\lambda_{\max}|\xi|^2
\end{align}
for any $x\in\Omega$ and $\xi\in\setR^n$, a local version of the result \eqref{eq:CZ} is proved in \cite{Me1} for $q\in[2,2+\epsilon)$ for some small $\epsilon>0$. To obtain the estimate for all $q\in(1,\infty)$, one needs additional regularity assumption on $\mathbb{A}$.
In \cite{D1} the assumption $\mathbb{A}\in\setVMO$ is made to prove~\eqref{eq:CZ} for $q\in(1,\infty)$ for $\partial \Omega \in C^{1,1}$ and in \cite{AusQaf02} for $\partial \Omega \in C^1$. A global result on $\setR^n$ is obtained in \cite{IS} and a local result for the case of systems is proved in \cite{DFZ} for~$\bbA \in \setVMO$. The condition $\mathbb{A}\in\setVMO$ is relaxed to a small $\setBMO$ condition. The global results for bounded domains are obtained in a series of papers \cite{B1,BCKW,BW1}.

In this paper we are also interested in the degenerate case, where~\eqref{eq:unif.ellip} fails. The most simple example of which is  $\bbA(x)=|x|^{\pm\epsilon}\identity$ with $\epsilon>0$ small. Instead of \eqref{eq:unif.ellip}, we assume
\begin{align}\label{eq:deg.ellip}
	\Lambda^{-2}\mu(x)|\xi|^2\leq\left<\mathbb{A}(x)\xi,\xi\right>\leq\mu(x)|\xi|^2
\end{align}
where $\mu(x):=\abs{\bbA(x)}=\abs{\bbM(x)}^2=\omega^2(x)$. In \cite{FKS}, it is proved that if $\mu$ belongs to the Muckenhoupt class $\mathcal{A}_2$, then the solution $u$ of \eqref{eq:lap} is H\"{o}lder continuous. Gradient estimates are obtained in \cite{CMP1} under \eqref{eq:deg.ellip}, $\mu\in\mathcal{A}_2$ and a smallness assumption in terms of a weighted $\setBMO$ norm of $\mathbb{A}$. They yield $|F|^q\mu\in L^{1}_{\loc}\Rightarrow |\nabla u|^q\mu\in L^{1}_{\loc}$ for all $q\in(1,\infty)$, including the case $\mu(x)=|x|^{\pm\epsilon}\identity$ for small $\epsilon>0$. The global result is obtained in \cite{P1} and the local result for the case of systems is proved in \cite{CMP2}. In the recent paper \cite{BDGN}, the authors prove a new type of gradient estimates with the implication that $(|F|\omega)^q\in L^{1}_{\loc}\Rightarrow (|\nabla u|\omega)^q\in L^{1}_{\loc}$ for all $q\in(1,\infty)$, assuming \eqref{eq:deg.ellip} and the smallness condition for the $\setBMO$ norm of $\log\mathbb{A}$ as follows:
\begin{align}\label{eq:logBMOA}
\sup_{B\Subset\Omega}\dashint_{B}\lvert\log\mathbb{A}(x)-\mean{\log\mathbb{A}}_{B}|\,dx\leq\frac{\delta}{q}
\end{align}
for some $\delta=\delta(n,p,\Lambda)$, where $\mean{f}_B=\dashint_{B}f\,dx$ for an integrable function $f:\setR^n\rightarrow\setR^{n\times n}_{\sym}$. Here, we can define $\log\mathbb{A}:\setR^n\rightarrow\setR^{n\times n}_{\sym}$, the logarithm of the matrix-valued weight $\mathbb{A}$, since $\mathbb{A}$ is positive definite almost everywhere. This novel $\log$--$\setBMO$ condition of \cite{BDGN} not only includes the degenerate weights of the form $\mathbb{A}(x)=|x|^{\pm\epsilon}\identity$ for small $\epsilon>0$, but also has the optimality in terms of the obtainable integrability exponent $q$. Compared to \cite{CMP1}, where $\mu \,dx$ is treated as a measure, the degenerate weight $\mu$ or better~$\omega$ in \cite{BDGN} plays the role of a multiplier.
Also in this paper we treat~$\omega$ as a multiplier, which seems also important for the optimal dependency of~$q$ on the constants.
We compare the conditions used in \cite{CMP1} and those introduced in \cite{BDGN} in Section~\ref{ssec:comparison-CMP}.


Now, consider the $p$-Laplacian case. If we write
$A(\xi):=|\xi|^{p-2}\xi$ and $\mathcal{A}(x,\xi):=|\setM(x)\xi|^{p-2}\setM^2(x)\xi$, then \eqref{eq:plap} is equivalent to
\begin{align}\label{eq:plap1}
  -\divergence\mathcal{A}(\cdot,\nabla u)=-\divergence\mathcal{A}(\cdot,F).
\end{align}
Writing $\setM^2=\mathbb{A}$ and $\mathcal{A}(\cdot,F)=G$ for $\mathbb{A}:\Omega\rightarrow\setR^{n\times n}_{\sym}$ and $G\,:\, \Omega \to \setR^n$ we can write \eqref{eq:plap1} as
\begin{align}\label{eq:plap2}
  -\divergence\left(\left<\mathbb{A}\nabla u,\nabla u\right>^{\frac{p-2}{2}}\mathbb{A}\nabla u\right)=-\divergence G.
\end{align}
Then $u$ is the minimizer of the following functional:
\begin{align*}
  \mathcal{P}(v)&:=\frac{1}{p}\int_{\Omega}\left<\mathbb{A}\nabla v,\nabla v\right>^{\frac{p}{2}}\,dx-\int_{\Omega}G \cdot\nabla v\,dx\\
                &=\frac{1}{p}\int_{\Omega}|\setM\nabla v|^p\,dx-\int_{\Omega}|\setM F|^{p-2}\setM F\cdot(\setM\nabla v)\,dx.
\end{align*}

If $p\in(1,\infty)$ and $\setM=\identity$, then $A(\nabla u)=\mathcal{A}(\cdot,\nabla u)$. In this case, the H\"{o}lder continuity of $u$ and $\nabla u$ is investigated in \cite{L0,U1}, and the gradient regularity estimates were obtained in \cite{DM,I1}. In the recent years, there have been many research activities for the gradient estimates in terms of $A(\nabla u)$. The $\setBMO$ type estimate with the implication that $G\in\setBMO \Rightarrow A(\nabla u)\in\setBMO$ is shown in \cite{DM} for $p>2$ and \cite{DKS} for $1<p<\infty$. In \cite{DKS}, the implication $G\in C^{0,\alpha}\Rightarrow A(\nabla u)\in C^{0,\alpha}$ for small $\alpha>0$ is proved. A local pointwise estimate is proved in \cite{BCDKS} and extended to the global one in \cite{BCDS}. Estimates in Besov space and Triebel-Lizorkin spaces up to differentiability one for $n=2$ and $p>2$ are shown in \cite{BDW}. Besov space regularity for $\nabla u$ is also considered in \cite{BCGOP,CGP}. The result $A(\nabla u)\in W^{1,2}$ when $\divergence G \in L^2$ is obtained in \cite{CM1} for scalar equations for $p>1$ and for vectorial systems in \cite{CM2} for $p> \frac 32$ and for $p> 2(2-\sqrt{2})\approx 1.1715$ in~\cite{BCDM}. Gradient potential estimates are studied for equations in \cite{KuuMin12, KuuMin13} and for systems in \cite{BY,DM1,KuuMin14,KuuMin18}.

Now, we pay attention to the weighted case. The local version of \eqref{eq:CZ} is proved for $1<p<\infty$, with a uniformly elliptic weight $\setM$ as in~\eqref{eq:unif.ellip} with $\setM\in\setVMO$ in \cite{KZ}. Since~$\setM$ is uniformly elliptic, we have $\omega(x) \eqsim 1$, so the results reduce to the transfer of~$L^q$-regularity from $F$ to $\nabla u$.  The global estimate is obtained in \cite{KZ1} with a $C^{1,\alpha}$ domain $\Omega$ for $\alpha\in(0,1]$. The assumption $\setM\in\setVMO$ has been weakened to the one that $\setM$ has a small $\setBMO$-norm, as shown in \cite{BW2,BWZ,C2,LCZ}. Under similar assumptions it is possible to replace the $L^q$-regularity transfer by $L^q(\sigma\,dx)$-regularity transfer for suitable Muckenhoupt weights~$\sigma$, see \cite{BR,MP0,MP1,P0}. Note that the weight~$\sigma$ is not related to the weight~$\omega$ of the equation.
  
Now, we introduce Lipschitz domains along with our optimal regularity assumption for the boundary of the domain.
\begin{definition}\label{def:lip}
  Let $\delta\in[0,\frac{1}{2n}]$ and $R>0$ be given. Then $\Omega$ is called $(\delta,R)$--Lipschitz if for each $x_0\in\partial\Omega$, there exists a coordinate system $\{x_1,\dots,x_n\}$ and Lipschitz map $\psi:\setR^{n-1}\rightarrow\setR$ such that $x_0=0$ in this coordinate system, and there holds
  \begin{align}\label{eq:lip1}
    \Omega\cap B_R(x_0)=\{x=(x_1,\dots,x_n)=(x',x_n)\in B_R(x_0):x_n>\psi(x')\}
  \end{align}
  and
  \begin{align}\label{eq:psi}
    \|\nabla\psi\|_{\infty}\leq\delta.
  \end{align}
\end{definition}
Imposing a Lipschitz condition for the boundary of the domain appears in many papers, the regularity and the asymptotic behavior of caloric function \cite{ACS}, homogenization \cite{KS}, oblique derivative problem \cite{L01,L02,L3}, H\"{o}lder continuity of solutions for Robin boundary condition \cite{N1}, regularity results for elliptic Dirichlet problem \cite{S1}, Calder\'{o}n-Zygmund estimates \cite{B1,BW2}. We would like to point out that in \cite{KZ1} $C^{1,\alpha}$ regularity with $\alpha\in(0,1]$ is assumed for $\partial\Omega$. It was observed in~\cite{AusQaf02} (in the linear case) that $\partial \Omega \in C^1$ is enough.
Our Lipschitz assumption for the boundary is weaker than both $C^{1,\alpha}$ and $C^1$ assumption on $\partial \Omega$, so Theorem~\ref{thm:lin} and~\ref{thm:nonlin} can be both applied in particular to~$C^{1,\alpha}$ and $C^1$-domains. Our assumption is indeed an optimal one to be discussed in Section~\ref{ssec:sharpness}. The sharp relation between the smallness parameter of the boundary and the integrability exponent~$q$ is,
as far as we know, new in the literature, even in the unweighted, linear case.



Our optimal regularity assumption for $\setM$ is a small $\setBMO$ assumption on its logarithm. This condition is also used in~\cite{BDGN} for the interior estimates.
\begin{definition}
  We say that $\log\setM$ is $(\delta,R)$-vanishing if
  \begin{align}\label{eq:logBMO}
    \lvert\log\setM|_{\setBMO(\setR^n)}:=\sup_{y\in\setR^n}\sup_{0<r\leq R}\dashint_{B_{r}(y)}\lvert\log\setM(x)-\mean{\log\setM}_{B_r(y)}|\,dx\leq\delta.
  \end{align}
\end{definition}

Now, we state the main theorems.
\begin{theorem}[Linear case]\label{thm:lin}
  Define $\omega$ as \eqref{eq:omega}, and assume \eqref{eq:M} and $F\in L^{q}_{\omega}(\Omega)$ for $q\in(1,\infty)$ in \eqref{eq:lap}. Then there exists a constant $\delta=\delta(n,\Lambda)\in(0,\frac{1}{2})$ such that if for some $R$,
  \begin{subequations}
    \begin{align}\label{eq:small}
      \log\mathbb{A}\,\,\,\text{is}\,\,\, \left(\delta\min\left\{\frac{1}{q},1-\frac{1}{q}\right\},R\right)\text{--vanishing and}
    \end{align}
    \begin{align}\label{eq:Omegasmall}		
      \Omega\,\,\text{is}\,\,\left(\delta\min\left\{ \frac{1}{q},1-\frac{1}{q}\right\},R\right)\text{--Lipschitz},
    \end{align}  
  \end{subequations}
  then the weak solution $u\in W^{1,2}_{0,\omega}(\Omega)$ of \eqref{eq:lap} satisfies $\nabla u\in L^{q}_{\omega}(\Omega)$ and we have the estimate
  \begin{align}\label{eq:lin}
    \int_{\Omega}(|\nabla u|\omega)^q\,dx\leq c\int_{\Omega}(|F|\omega)^q\,dx
  \end{align}
  for some $c=c(n,\Lambda,\Omega,q)$.
\end{theorem}

For the non-linear case, we have the following result.

\begin{theorem}[Non-linear case]\label{thm:nonlin}
Define $\omega$ as \eqref{eq:omega}, and assume \eqref{eq:M} and $F\in L^{q}_{\omega}(\Omega)$ for $q\in [p,\infty)$ in \eqref{eq:plap}. Then there exists a constant $\delta=\delta(n,p,\Lambda)\in(0,\frac{1}{2})$ such that 
if for some $R$,
\begin{align}\label{eq:small2}
\log\setM\,\,\,\text{is}\,\,\,\left(\dfrac{\delta}{q},R\right)\text{--vanishing and }\Omega\,\,\text{is}\,\,\left(\dfrac{\delta}{q},R\right)\text{--Lipschitz},
\end{align}   
then the weak solution $u\in W^{1,p}_{0,\omega}(\Omega)$ of \eqref{eq:plap} satisfies $\nabla u\in L^{q}_{\omega}(\Omega)$ and we have the estimate
\begin{align}\label{eq:nonlin}
\int_{\Omega}(|\nabla u|\omega)^q\,dx\leq c\int_{\Omega}(|F|\omega)^q\,dx
\end{align}
for some $c=c(n,p,\Lambda,\Omega,q)$.
\end{theorem}

In principle we  use a standard  perturbation argument combined with the regularity of $p$-harmonic functions. This argument is for example developed in \cite{Iwa82} and  \cite{CP}, and used in \cite{KZ1}. However, we modify this technique such that it is possible to obtain optimal estimates in terms of the smallness of oscillation parameter $\lvert\log\setM|_{\setBMO}$ and the boundary regularity parameter $\|\nabla\psi\|_{\infty}$. In particular, we obtain a linear dependence for the reciprocal of the integrability exponent instead of an exponential one. This is one of the main novelties of this paper.

The approach from~\cite{Iwa82} and~\cite{CP} can be reduced to redistributional estimates in terms of maximal operator of the gradient. However this technique always introduces an exponential dependence of~$q$ on the smallness parameter~$\delta$. We avoid this problem by using a qualitative  version of  the global Fefferman-Stein inequality  $\|f\|_{L^q(\setR^n)}\leq c q \|\mathcal{M}^{\sharp}_1 f\|_{L^q(\setR^n)}$ (see \eqref{eq:def.max} in Section 3). The important feature is   the linear dependency on the exponent $q$. This allows us  to extract the sharp dependency of $\lvert\log\setM|_{\setBMO}$ and $\|\nabla\psi\|_{\infty}$.

The interior maximal regularity with optimal constants was already described in \cite{BDGN}. In this paper we extend those results up to the boundary with an optimal dependence and the boundary parameters. To this end, we first use the localization argument adapted to our boundary comparison estimate and provide the pointwise sharp maximal function estimate for the localized function of $u$. As an auxiliary step we provide $C^{1,\alpha}$-regularity and the decay estimates up to boundary for the solutions of the reference problems. To this end, we employ the reflection principle of the reference problems, which is one of the intrinsic property in the divergence type equation, see \cite{M1}.



This paper is organized as follows: Section 2 is devoted to the basic notation and results which will be frequently used throughout the paper. In Section 3 we provide the proof of Theorem \ref{thm:lin} and Theorem \ref{thm:nonlin}. In Section 4 we discuss the sharpness of our smallness assumptions \eqref{eq:small} and \eqref{eq:small2}.

\section{Notation and Preliminary Results}

Denote $B_{r}(x)\subset\setR^n$ by the open ball of radius $r>0$ and center $x\in\setR^n$. For a ball $B$, let $r_B$ be the radius and $x_B$ be the center of $B$. For $x=(x_1,\dots,x_n)$ write $B_r^+(x)=B_r(x)\cap\{y=(y_1,\dots,y_n)\in\setR^n:y_n\geq x_n\}$. For an open set $U$ and a function $f$ we abbreviate $\mean{f}_{U}:=\dashint_{U}f(x)\,dx$. We write $\chi_U$ for the characteristic function of the set $U$. Let us denote $c$ by a general constant depending only on $n,p$ and $\Lambda$ which may vary on different occasions. We also write $f\lesssim g$ when $f\leq c g$, and write $f\eqsim g$ when $f\lesssim g$ and $g\lesssim f$ hold. For $1<p<\infty$, $p'=\frac{p}{p-1}$ means the conjugate exponent of $p$.

\subsection{Matrix-valued Weights and Logarithms}

We write $\setR^{n\times n}_{\sym}$ for symmetric, real-valued matrices. We denote $\setR^{n\times n}_{\geq 0}$ by the cone of symmetric, real-valued and positive semidefinite matrices. The collection of positive definite matrices is denoted by $\setR^{n\times n}_{>0}$. For $\mathbb{X}, \mathbb{Y}\in\setR^{n\times n}_{\sym}$, we write $\mathbb{X}\geq\mathbb{Y}$ provided $\mathbb{X}-\mathbb{Y}\in\setR^{n\times n}_{\geq 0}$.

Let $\setM\colon\setR^n\rightarrow\setR^{n\times n}_{\geq 0}$ be a (matrix-valued) weight if $\setM$ is positive definite a.e., and $\omega\colon\setR^n\rightarrow[0,\infty)$ be a (scalar) weight if $\omega$ is positive a.e.. For $\mathbb{L}\in\setR^{n\times n}$, let $|\mathbb{L}|$ denote the spectral norm, which means $|\mathbb{L}|=\sup_{|\xi|\leq 1}|\mathbb{L}\xi|$. If $\bbL$ is symmetric, then $\abs{\bbL} = \sup_{\abs{\xi} \leq 1} \skp{\bbL \xi}{\xi}$.

We consider the matrix exponential $\exp\colon \setR^{n\times n}_{\sym}\rightarrow\setR^{n\times n}_{>0}$, with its unique inverse mapping $\log\colon\setR^{n\times n}_{>0}\rightarrow\setR^{n\times n}_{\sym}$. Thus, we can define $\log\setM\colon\setR^n\rightarrow\setR^{n\times n}_{\sym}$, since $\setM\colon\setR^n\rightarrow\setR^{n\times n}_{\sym}$ is positive definite a.e. We now define the logarithmic means 
\begin{align*}
\mean{\omega}^{\log}_{U}&:=\exp(\mean{\log\omega}_U),\\
\mean{\setM}^{\log}_{U}&:=\exp(\mean{\log\setM}_U)
\end{align*}
for some subset $U\subset\setR^n$. The logarithmic mean has the following compatibility property under taking reciprocal:
\begin{align*}
\left<\dfrac{1}{\omega}\right>^{\log}_{U}=\exp(-\mean{\log\omega}_{U})=\dfrac{1}{\mean{\omega}^{\log}_{U}}.
\end{align*}
Moreover, using the relations $\log(\setM^{-1})=-\log\setM$ and $(\exp(\mathbb{L}))^{-1}=\exp(-\mathbb{L})$, we also obtain
\begin{align*}
\mean{\setM^{-1}}^{\log}_{U}=\exp(-\mean{\log(\setM)}_U)=\left(\exp(\mean{\log(\setM)}_{U})\right)^{-1}=(\mean{\setM}^{\log}_{U})^{-1}.
\end{align*}

\subsection{Weighted Lebesgue Space and Muckenhoupt Weights}

For $1<p<\infty$ and a weight $\omega\in L^p_{\loc}(\setR^n)$ with $\omega^{-1}\in L^{p'}_{\loc}(\setR^n)$, we define the weighted Lebesgue spaces
\begin{align*}
L^p_{\omega}(\setR^n):=\{f:\setR^n\rightarrow\setR\colon\omega f \in L^{p}(\setR^n)\}
\end{align*}
equipped with the norm $\|f\|_{p,\omega}:=\|f\omega\|_p$. In particular, we treat the weight~$\omega$ as a multiplier. The dual space of $L^{p}_{\omega}(\setR^n)$ is $L^{p'}_{1/\omega}(\setR^n)$. Both $L^{p}_{\omega}(\setR^n)$ and $L^{p'}_{1/\omega}(\setR^n)$ are Banach spaces and continuously embedded into $L^1_{\loc}(\setR^n)$. Let $W^{1,p}_{\omega}(\Omega)$ be the weighted Sobolev space which consists of functions $u\in W^{1,1}(\Omega)$ such that $u,|\nabla u|\in L^p_{\omega}(\Omega)$, equipped with the norm $\|u\|_{W^{1,p}_{\omega}(\Omega)}=\|u\|_{L^p_{\omega}(\Omega)}+\|\nabla u\|_{L^p_{\omega}(\Omega)}$. Let $W^{1,p}_{0,\omega}(\Omega)$ denote the subspace of $W^{1,p}_{\omega}(\Omega)$ of functions with zero traces on $\partial\Omega$.

For $1<p<\infty$, a weight $\mu\in L^{1}_{\loc}(\setR^n)$ is called an $\mathcal{A}_p$-Muckenhoupt weight if
\begin{align*}
  [\mu]_{\mathcal{A}_p}:= \sup_{B\subset\setR^n}\left(\dashint_{B}\mu\,dx\right)\left(\dashint_{B}\mu^{-\frac{1}{p-1}}\,dx\right)^{p-1}<\infty.
\end{align*}
If $\mu$ is an $\mathcal{A}_p$-Muckenhoupt weight, then the maximal operator is bounded on $L^{p}(\setR^n,\mu)$ for $1<p<\infty$. We point out some properties for a Muckenhoupt weight related to its logarithmic means. If $\omega^p$ is an $\mathcal{A}_p$-Muckenhoupt weight, then from Jensen's inequality,
\begin{align}
\begin{split}
\label{eq:omegalog}
\left(\dashint_{B}\omega^p\,dx\right)^{\frac{1}{p}}&\leq [\omega^p]^{\frac{1}{p}}_{\mathcal{A}_p}\mean{\omega}^{\log}_{B},\\
\left(\dashint_{B}\omega^{-p'}\,dx\right)^{\frac{1}{p'}}&\leq [\omega^p]^{\frac{1}{p}}_{\mathcal{A}_p}\mean{\omega^{-1}}^{\log}_{B}=\dfrac{[\omega^p]^{\frac{1}{p}}_{\mathcal{A}_p}}{\mean{\omega}^{\log}_{B}}.
\end{split}
\end{align} 
Conversely, if \eqref{eq:omegalog} holds, then $\omega^p$ is an $\mathcal{A}_p$-Muckenhoupt weight, since $\mean{\omega}^{\log}_{B}\mean{\omega^{-1}}^{\log}_{B}=1$ holds.

\subsection{Parametrization of the Boundary}
Let $\Omega$ be $(\delta,4R)$--Lipschitz and $x_0\in\partial\Omega$. Then there exists a Lipschitz map $\psi:\setR^{n-1}\rightarrow\setR$ as in Definition \ref{def:lip}. By translation, without loss of generality we assume $x_0=0$ and $\psi(0)=0$. We define $\Psi:\setR^n\rightarrow \setR^n$ as 
\begin{align}\label{eq:defPsi}
\Psi(x',x_n)=(x',x_n-\psi(x'))\quad\text{for }(x',x_n)\in\setR^n
\end{align}
and so there hold $\Psi(\partial\Omega\cap B_{4R}(0))\subset\{(y',y_n):y_n=0\}$, $\Psi(\overline{\Omega}\cap B_{4R}(0))\subset\{(y',y_n):y_n\geq0\}$ and $\Psi(0)=0$. The mapping $\Psi$ is invertible, with a Lipschitz continuous inverse $\Psi^{-1}$. We easily obtain
\begin{align}\label{eq:J}
\nabla\Psi(x',x_n)=\begin{pmatrix}
I & 0 \\
-\nabla\psi(x') & 1 
\end{pmatrix}=\begin{pmatrix}
1 & 0 & \dots & 0\\
0 & \ddots &  & \vdots\\
\vdots &  &  1  &  0\\
-\partial_{x_1}\psi(x') & \dots  &  -\partial_{x_{n-1}}\psi(x')  &  1
\end{pmatrix}
\end{align}
for $(x',x_n)\in\setR^n$, where the right-hand side of \eqref{eq:J} is an $n\times n$ matrix. In particular, $\det(\nabla\Psi(x))=1$ and so $|\Psi(B)|=|B|$ for each ball $B\subset\setR^n$. Note that $|\identity-(\nabla\Psi)(x)|\leq n\|\nabla\psi\|_{\infty}$.

Now, we provide some geometric properties related to the maps $\psi$ and $\Psi$ which will be used throughout the paper.
\begin{remark}\label{rmk:Psi}
From now on, we implicitly use the following properties. If we assume
\begin{align}\label{eq:Omega}
\Omega\,\,\text{ is }\,\,(\delta,4R)\text{--Lipschitz with }\,\,\delta\in\left[0,\frac{1}{2n}\right]\,\,\text{ and }\,\,R>0,
\end{align}
Then for any induced map $\Psi$ from the Lipschitz map $\psi$ assigned to given $x_0\in\partial\Omega$, we have $\abs{\identity - \nabla \psi}\leq  \frac 12$,
\begin{align}\label{eq:Psi.ball}
  \tfrac{1}{2}B\subset\Psi(B)\subset 2B
\end{align}
for all ball $B\subset\setR^n$, and the following measure density properties also hold:
\begin{align}\label{eq:m.den1}
  \sup_{0<r\leq 4R}\sup_{y\in\Omega}\dfrac{|B_r(y)|}{|\Omega\cap B_r(y)|}\leq 4^n
\end{align}
and
\begin{align}\label{eq:m.den2}
  \inf_{0<r\leq 4R}\inf_{y\in\partial\Omega}\dfrac{|B_r(y)\cap\Omega^c|}{|B_r(y)|}\geq 4^{-n}.   \end{align}
\end{remark}

\subsection{Preliminary Estimates}
\label{sec:prel-estim}

We first consider a weighted Poincar\'{e} inequality with partial zero boundary values. The corresponding mean value version is given in~\cite{BDGN}.

\begin{proposition}[Weighted Poincar\'{e} inequality at boundary]\label{prop:SP}
  Let $1<p<\infty$ and $\theta\in(0,1)$ be such that $\theta p\geq\max\set{1,\frac{np}{n+p}}$. Moreover, let $B_r=B_r(x_0)$ with $x_0\in\Omega$ and $B_{\frac{3}{2}r}(x_0)\not\subset\Omega$. Assume that $\Omega$ is $(\delta,3r)$-Lipschitz with $\delta\in[0,\frac{1}{2n}]$ and that $\omega$ is a weight on $B_{3r}$ with
  \begin{align}
    \label{eq:omega1}
    \sup_{B'\subset B_{3r}}\left(\dashint_{B'}\omega^p\,dx\right)^{\frac{1}{p}}\left(\dashint_{B'}\omega^{-(\theta p)'}\,dx\right)^{\frac{1}{(\theta p)'}}\leq c_1.
  \end{align}
  Then for any $v\in W^{1,p}_{\omega}(B_{2r}\cap\Omega)$ with $v=0$ on $\partial\Omega\cap B_{2r}$,
  \begin{align}
    \label{eq:sp2}
    \left(\dashint_{B_{2r}\cap\Omega}\left|\dfrac{v}{r}\right|^p\omega^p\,dx\right)^{\frac{1}{p}}\leq c\left(\dashint_{B_{2r}\cap\Omega}(|\nabla v|\omega)^{\theta p}\,dx\right)^{\frac{1}{\theta p}}
  \end{align}
  holds with $c=c(n,p,c_1)$.

\end{proposition}
\begin{proof}
  Since $v\in W^{1,p}_{\omega}(\Omega_{2r})$ with $v=0$ on $\partial\Omega\cap B_{2r}$, we can take the zero extension of $v$ on the set $B_{2r}\setminus\Omega$. Since $B_{\frac{3r}{2}}(x_0)\not\subset\Omega$, by \eqref{eq:m.den1} and \eqref{eq:m.den2} in Remark \ref{rmk:Psi}, for $A:=B_{2r}\setminus\Omega$, $|A|\eqsim|B_{2r}\cap\Omega|$ holds. Then $\mean{v}_{A}=0$ holds, and so by Remark \ref{rmk:Psi}, Proposition 3 in \cite{BDGN} and Jensen's inequality, we have
  \begin{align*}
    \left(\dashint_{B_{2r}\cap\Omega}\left|\dfrac{v}{r}\right|^p\omega^p\,dx\right)^{\frac{1}{p}}
    &\lesssim \left(\dashint_{B_{2r}}\left|\dfrac{v-\mean{v}_{B_{2r}}}{r}\right|^p\omega^p\,dx\right)^{\frac{1}{p}}+\left(\dashint_{B_{2r}}\left|\dfrac{\mean{v}_{B_{2r}}-\mean{v}_A}{r}\right|^p\omega^p\,dx\right)^{\frac{1}{p}}
    \\
    &\lesssim \left(\dashint_{B_{2r}}(|\nabla v|\omega)^{\theta p}\,dx\right)^{\frac{1}{\theta p}}+\left[\dashint_{B_{2r}}\left(\dashint_{A}\left|\dfrac{v(y)-\mean{v}_{B_{2r}}}{r}\right|\,dy\right)^p\omega^p\,dx\right]^{\frac{1}{p}}.
  \end{align*}
  Here, using $|A|\eqsim|B_{2r}\cap\Omega|\eqsim|B_{2r}|$, H\"{o}lder's inequality, \eqref{eq:omega1} and Proposition 3 in \cite{BDGN}, it follows that
  \begin{align*}
    \lefteqn{\dashint_{B_{2r}}
    \left(\dashint_{A}\left|\dfrac{v(y)-\mean{v}_{B_{2r}}}{ r}\right|\,dy\right)^p\omega(x)^p\,dx } \qquad
    &
    \\
    &\quad\quad\lesssim \dashint_{B_{2r}}\left[\left(\dashint_{B_{2r}}\left|\dfrac{v(y)-\mean{v}_{B_{2r}}}{r}\right|^p\omega(y)^p\,dy\right)^{\frac{1}{p}}\left(\dashint_{B_{2r}}\omega(y)^{-p'}\,dy\right)^{\frac{1}{p'}}\right]^p\omega(x)^p\,dx\\
    &\quad\quad\eqsim\left(\dashint_{B_{2r}}\left|\dfrac{v-\mean{v}_{B_{2r}}}{r}\right|^p\omega^p\,dy\right)\left(\dashint_{B_{2r}}\omega^{-p'}\,dy\right)^{p-1}\left(\dashint_{B_{2r}}\omega^p\,dx\right)\\
    &\quad\quad\lesssim\left(\dashint_{B_{2r}}(|\nabla v|\omega)^{\theta p}\,dx\right)^{\frac{1}{\theta}}.
  \end{align*}
  Now, since $v=|\nabla v|=0$ on $B_{2r}\setminus\Omega$ and $|B_{2r}\cap\Omega|\eqsim|B_{2r}|$, we have \eqref{eq:sp2}.
\end{proof}

Let us collect a few auxiliary results from~\cite{BDGN} that will be used later.
It follows from~\cite[(3.24)]{BDGN} with $\omega = \abs{\bbM}$ that
\begin{align}
  \label{eq:aux-mean1}
  \begin{aligned}
    \Lambda^{-1} \logmean{\abs{\bbM}}_B &\leq \abs{\logmean{\bbM}_B} \leq \logmean{\abs{\bbM}}_B,
    \\
    \Lambda^{-2} \logmean{\abs{\bbA}}_B &\leq \abs{\logmean{\bbA}_B} \leq \logmean{\abs{\bbA}}_B.
  \end{aligned}
\end{align}
Moreover, by monotonicity of the scalar versions of $\exp$ and $\log$ we have
\begin{align}
  \label{eq:aux-mean2}
  \logmean{\abs{\bbA}}_B &\leq \mean{\abs{\bbA}}_B.
\end{align}

\begin{lemma}{\cite[Lemma 4]{BDGN}}\label{lem:M.trans.omega}
  For a matrix-valued weight $\mathbb{M}$ and $\omega=|\mathbb{M}|$ we have
  \begin{align*}
    \dashint_{B}\lvert\log\omega(x)-\left<\log\omega\right>_{B}|\,dx\leq 2\dashint_{B}\lvert\log\mathbb{M}(x)-\left<\log\mathbb{M}\right>_{B}|\,dx
  \end{align*}
  and so $\lvert\log\omega|_{\setBMO(B)}\leq 2\lvert\log\mathbb{M}|_{\setBMO(B)}$.
\end{lemma}
The next results provides a qualitative John-Nirenberg type inequality.
\begin{lemma}{\cite[Proposition 5]{BDGN}}\label{lem:BMOM}
  There exist constants $\kappa_1=\kappa_1(n,\Lambda)>0$ and $c_2=c_2(n,\Lambda)>0$ such that the following holds: If $t\geq 1$ and $\setM$ is a matrix-valued weight with $\lvert\log \setM|_{\setBMO(B)}\leq\frac{\kappa_1}{t}$, then we have
  \begin{align*}
    \left(\dashint_{B}\left(\dfrac{|\setM(x)-\mean{\setM}^{\log}_{B}|}{|\mean{\setM}^{\log}_{B}|}\right)^t\,dx\right)^{\frac{1}{t}}\leq c_2t\lvert\log \setM|_{\setBMO(B)}.
  \end{align*}
  The same holds with $\omega$ instead of $\setM$.
\end{lemma}
The following results is a minor modification of~\cite[Proposition 6]{BDGN}.
\begin{lemma}\label{lem:Mu.weight}
  Let $\kappa_1$ and $c_2$ be as in Lemma \ref{lem:BMOM}. Then with a constant $\beta=\beta(n,\Lambda)=\min\left\{\kappa_1,1/c_2\right\}>0$, the following holds for all weights $\omega$.
  \begin{enumerate}
  \item \label{itm:Mu.weight-1} If $\lvert\log\omega|_{\setBMO(B)}\leq\frac{\beta}{\gamma}$ with $\gamma\geq1$, then there holds
    \begin{align*}
      \left(\dashint_{B}\omega^{\gamma}\,dx\right)^{\frac{1}{\gamma}}\leq 2\mean{\omega}^{\log}_{B}\quad\text{and}\quad \left(\dashint_{B}\omega^{-\gamma}\,dx\right)^{\frac{1}{\gamma}}\leq 2\dfrac{1}{\mean{\omega}^{\log}_{B}}.
    \end{align*}
  \item \label{itm:Mu.weight-2} If $\lvert\log\omega|_{\setBMO(B)}\leq\beta\min\{\frac{1}{p},\frac{1}{p'}\}$ with $1<p<\infty$, then $\omega^p$ is an $\mathcal{A}_p$--Muckenhoupt weight and
    \begin{align*}
      [\omega^p]^{\frac{1}{p}}_{\mathcal{A}_p}=\sup_{B'\subset B}\left(\dashint_{B'}\omega^p\,dx\right)^{\frac{1}{p}}\left(\dashint_{B'}\omega^{-p'}\,dx\right)^{\frac{1}{p'}}\leq 4.
    \end{align*}
  \item \label{itm:Mu.weight-3}Let $1<p<\infty$ and $\theta\in(0,1)$ be such that $\theta p>1$. If $\lvert\log\omega|_{\setBMO(B)}\leq\beta\min\{\frac{1}{p},1-\frac{1}{\theta p}\}$, then
    \begin{align*}
      \sup_{B'\subset B}\left(\dashint_{B'}\omega^p\,dx\right)^{\frac{1}{p}}\left(\dashint_{B'}\omega^{-(\theta p)'}\,dx\right)^{\frac{1}{(\theta p)'}}\leq 4.
    \end{align*}
  \end{enumerate}
\end{lemma}
\begin{proof}
  The proof is the same as in \cite[Proposition 6]{BDGN} with minimal changes due to the localized versions.
\end{proof}
\begin{remark}
  Using the relation $\log(\setM^{-1})=-\log(\setM)$ and $\log(\omega^{-1})=-\log(\omega)$, we can apply Lemma \ref{lem:BMOM} and Lemma \ref{lem:Mu.weight} also to $\setM^{-1}$ and $\omega^{-1}$.
\end{remark}

\subsection{Summary on Orlicz Functions} We say that a function $\tilde{\phi}:\setR_{\geq 0}\rightarrow\setR_{\geq 0}$ is an N-function if $\tilde{\phi}':\setR_{\geq 0}\rightarrow\setR_{\geq 0}$ is a right-continuous and non-decreasing function with $\tilde{\phi}'(0)=0$ and $\lim_{t\rightarrow\infty}\tilde{\phi}'(t)=\infty$ such that $\tilde{\phi}(t)=\int^{t}_{0}\tilde{\phi}'(\tau)\,d\tau$. An N-function is called to satisfy the $\Delta_2$-condition if there is a constant $c>1$ such that $\tilde{\phi}(2t)\leq c\tilde{\phi}(t)$.

We now define a specific N-function
\begin{align*}
  \phi(t):=\dfrac{1}{p}t^p.
\end{align*}
Then we denote
\begin{align*}
  A(\xi)&:=\dfrac{\phi'(|\xi|)}{|\xi|}\xi=|\xi|^{p-2}\xi,
  \\ V(\xi)&:=\sqrt{\dfrac{\phi'(|\xi|)}{|\xi|}}\xi=|\xi|^{\frac{p-2}{2}}\xi.
\end{align*}
Let $\tilde{\phi}^*$ be the conjugate of an N-function $\tilde{\phi}$ as follows:
\begin{align*}
  \tilde{\phi}^*(t):=\sup_{s\geq 0}(ts-\tilde{\phi}(s)),\quad t\geq 0,
\end{align*}
and so $\phi^*(t)=\frac{1}{p'}t^{p'}$.

We also need the shifted N-functions as introduced in \cite{DE,DieKre08,DFTW,BDGN}. For $a\geq 0$ we define $\phi_a$ as
\begin{align}\label{eq:shifted}
  \phi_a(t):=\int^{t}_0\dfrac{\phi'(a\vee s)}{a\vee s}s \,ds.
\end{align} 
Here  $s_1\vee s_2:=\max\{s_1,s_2\}$ for $s_1,s_2\in\setR$. We call $a$ the \textit{shift}. So for $t \leq a$, then function~$\phi_a(t)$ is quadratic in~$t$.
One can see that $\phi_0=\phi$ holds, and $a\eqsim b$ implies $\phi_a(t)\eqsim \phi_b(t)$. Also, we have
\begin{align}
\phi_a(t)&\eqsim(a\vee t)^{p-2}t^2,\label{eq:phi}\\
\phi'_a(t)&\eqsim (a\vee t)^{p-2}t,\label{eq:phi2}\\
(\phi_a)^*&=(\phi^*)_{\phi'(a)},\label{eq:phistar}\\
(\phi_{|\xi|})^*&=(\phi^*)_{\phi'(|\xi|)}\label{eq:phistar2}
\end{align}
with constants depending only on $p$. Moreover, for $a\geq 0$, the collection of $\phi_a$ and $(\phi_a)^*$ satisfy the $\Delta_2$-condition with a $\Delta_2$-constant independent of $a$. 

We also have Young's inequality. For every $\epsilon>0$ there exists $c(\epsilon)=c(\epsilon,p)\geq1$ such that for all $s,t,a\geq 0$
\begin{align}\label{eq:st}
  st\leq c(\epsilon)(\phi_a)^*(s)+\epsilon\phi_a(t).
\end{align}
Here, $c(\epsilon)\eqsim \max \set{\epsilon^{-1},\epsilon^{-\frac{1}{p-1}}}$. Similarly, considering the relations $\phi_a(t)\eqsim\phi'_a(t)t$ and $(\phi_a)^*\eqsim(t\phi'_a(t))$, we have
\begin{align}\label{eq:st2}
  \begin{split}
    \phi'_a(s)t\leq c(\epsilon)\phi_a(s)+\epsilon\phi_a(t),\\\phi'_a(s)t\leq \epsilon\phi_a(s)+\tilde{c}(\epsilon)\phi_a(t)
  \end{split}
\end{align}
for all $s,t,a\geq 0$, where $\tilde{c}(\epsilon)\eqsim \max \set{\epsilon^{-1}, \epsilon^{1-p}}$.
Moreover, the following relation holds for $a\geq 0$:
\begin{align}\label{eq:equiv.phi}
\phi_a(\lambda a)\eqsim
\begin{cases}
\lambda^2\phi(a),\quad&\text{for }\lambda<1,\\\phi(\lambda a)\quad&\text{for }\lambda\geq 1.
\end{cases}
\end{align}
We emphasize the relation between $A, V$ and $\phi_a$ as in the following:

\begin{lemma}[{\cite[Lemma~41]{DFTW}}]\label{lem:equivA}
  For all $P,Q\in\setR^n$ we have
  \begin{align*}
    (A(P)-A(Q))\cdot(P-Q)\eqsim|V(P)-V(Q)|^2\eqsim\phi_{|Q|}(|P-Q|)\eqsim(\phi^*)_{|A(Q)|}(|A(P)-A(Q)|),
  \end{align*}
  \begin{align*}
    A(Q)\cdot Q=|V(Q)|^2\eqsim\phi_{|Q|}(|Q|)\eqsim\phi(|Q|)
  \end{align*}
  and
  \begin{align*}
    |A(P)-A(Q)|\eqsim(\phi_{|Q|})'(|P-Q|)\eqsim\phi'_{|P|\vee|Q|}(|P-Q|),
  \end{align*}
  where the implicit constants depend only on $p$.
\end{lemma}
We usually use the following \emph{change of shift}:
\begin{lemma}[{Change of shift, \cite[Corollary~44]{DFTW}}]\label{lem:change}
  For $\epsilon>0$, there exists $c_{\epsilon}=c_{\epsilon}(\epsilon,p)$ such that for all $P,Q\in\setR^n$ we have
  \begin{align*}
    \phi_{|P|}(t)&\leq c_{\epsilon}\phi_{|Q|}(t)+\epsilon|V(P)-V(Q)|^2,\\
    \left(\phi_{|P|}\right)^*(t)&\leq c_{\epsilon}\left(\phi_{|Q|}\right)^*(t)+\epsilon|V(P)-V(Q)|^2,
  \end{align*}
  where $c_{\epsilon}=c(\epsilon,p)$.
\end{lemma}

Also, we need the following \emph{removal of shift}.
\begin{lemma}[{Removal of shift, \cite[Lemma 13]{BDGN}}]\label{lem:removal}
  For all $a\in\setR^n$, $t\geq 0$ and $\epsilon\in(0,1]$, we have
  \begin{align}\label{eq:removal1}
    \phi'_{|a|}(t)\leq\phi'\left(\dfrac{t}{\epsilon}\right)\vee(\epsilon\phi'(|a|)),\\
    \phi_{|a|}(t)\leq\epsilon\phi(|a|)+c\epsilon\phi\left(\dfrac{t}{\epsilon}\right),\\
    (\phi_{|a|})^*(t)\leq\epsilon\phi(|a|)+c\epsilon\phi^*\left(\dfrac{t}{\epsilon}\right)
  \end{align}
  with $c=c(p)$.
\end{lemma}

\section{Global Maximal Regularity Estimates}
\label{sec:glob-cald-zygm}

We now provide global maximal regularity estimates for the weak solutions of our weighted $p$-Laplace equation for the linear case $p=2$ as well as the non-linear case $p\in(1,\infty)$. Let $\Omega\subset\setR^n$ be $(\delta,R)$--Lipschitz and $\bbM:\setR^n\to\setR^{n\times n}_{>0}$ be a degenerate elliptic matrix-valued weight with uniformly bounded condition number~\eqref{eq:M}. Recall, that~$\omega(x):=\abs{\bbM(x)}$.  Note, that \eqref{eq:M} is equivalent to
\begin{alignat}{2}
  \label{eq:M2}
  \Lambda^{-1}\omega(x)|\xi|&\leq|\setM(x)\xi|\leq\omega(x)|\xi|  & \quad\quad&\text{for all }\xi\in\setR^n
  \\
  \intertext{and also}
  \label{eq:M3}
  \Lambda^{-1}\omega(x)\text{Id}&\leq\setM(x)\leq\omega(x)\text{Id}
  &&\text{for all $x\in\Omega$.}
\end{alignat}
If we assume that $\log\setM$ has a small $\setBMO$-norm, i.e., assume
\begin{align}\label{eq:smallBMO.M2}
\lvert\log\setM|_{\setBMO(\Omega)}\leq\kappa,
\end{align}
we also have $\lvert\log\omega|_{\setBMO(\Omega)}\leq 2\kappa$ by Lemma \ref{lem:M.trans.omega}.
Suppose that $\kappa$ is so small such that by Lemma \ref{lem:Mu.weight}, $\omega^p$ is an $\mathcal{A}_p$-Muckenhoupt weight. Then $C^{\infty}_{0}(\Omega)$ is dense in $W^{1,p}_{0,\omega}(\Omega)$. Now, let $u\in W^{1,p}_{0,\omega}(\Omega)$ be the weak solution of \eqref{eq:plap} with $F\in L^{p}_{\omega}(\Omega)$, i.e., if we denote
\begin{align*}
\mathcal{A}(x,\xi):=|\setM(x)\xi|^{p-2}\setM^2(x)\xi=\setM(x) A(\setM(x)\xi),
\end{align*}
then
\begin{align}\label{eq:plap3}
\int_{\Omega}\mathcal{A}(x,\nabla u)\cdot \nabla \xi\,dx=\int_{\Omega}\mathcal{A}(x,F)\cdot\nabla \xi\,dx
\end{align}
for all $\xi\in W^{1,p}_{0,\omega}(\Omega)$. Since $\omega^p$ is an $\mathcal{A}_p$-Muckenhoupt weight, the existence  and uniqueness of $u$ is guaranteed by standard arguments from the calculus of variations.

\subsection{Caccioppoli and Reverse H\"{o}lder's Inequality}
\label{sec:cacc-reverse-hold}

We start with the standard Caccioppoli estimates associated with our degenerate $p$-Laplacian problem. We fix a ball $B_0:=B_R(x_0)$ with $x_0\in\partial\Omega$. Then since $\Omega$ is $(\delta,4R)$--Lipschitz, there exists a coordinate system $\{x_1,\dots,x_n\}$ such that $x_0=0$ in this coordinate system, and with the assigned Lipschitz map $\psi:\setR^{n-1}\rightarrow\setR$ we have \eqref{eq:lip1} with $4R$ instead of $R$. Let $u\in W^{1,p}_{\omega}(4B_0\cap\Omega)$ be a weak solution of
	
\begin{align}
  \label{eq:plap4}
  \begin{alignedat}{2}
    -\divergence\mathcal{A}(x,\nabla u) &=-\divergence\mathcal{A}(x,F) &\qquad& \text{in }4B_0\cap\Omega,
    \\
    u&=0&&\text{on }\partial\Omega\cap (4B_0).
  \end{alignedat}
\end{align}
From now on, let $B_r =B_r(\tilde{x})$ denote an arbitrary ball with $\tilde{x}\in\overline{\Omega}$ and $4B_r\subset 2B_0$. Denoting $a\Omega_{r}=aB_r\cap 2B_0\cap\Omega$ for $a\in\setR_+$, we have the following:

\begin{proposition}[Caccioppoli inequality]\label{prop:cacc} 
Let $u\in W^{1,p}_{\omega}(4B_0\cap\Omega)$ be a weak solution of \eqref{eq:plap4} and $B_r=B_r(\tilde{x})$ denote an arbitrary ball with $\tilde{x}\in\overline{\Omega}$ and $4B_r\subset 2B_0$.
\begin{itemize}
	\item[(1)](Interior case) If $2B_{r}\subset\Omega$, then we have
	\begin{align}\label{eq:cacc1}
	\dashint_{\Omega_{r}}|\nabla u|^p\omega^p\,dx\leq c \dashint_{2\Omega_{r}}\left|\dfrac{u-\mean{u}_{2\Omega_r}}{r}\right|^p\omega^p\,dx+c\dashint_{2\Omega_{r}}|F|^p\omega^p\,dx.
	\end{align}
	\item[(2)](Boundary case) Assume \eqref{eq:Omega}. If $2B_{r}\not\subset\Omega$, then we have
	\begin{align}\label{eq:cacc2}
	\dashint_{\Omega_{r}}|\nabla u|^p\omega^p\,dx\leq c \dashint_{2\Omega_{r}}\left|\dfrac{u}{r}\right|^p\omega^p\,dx+c\dashint_{2\Omega_{r}}|F|^p\omega^p\,dx.
	\end{align}
      \end{itemize}
      In both cases $c=c(n,p,\Lambda)$.
\end{proposition}
\begin{proof}
First, \eqref{eq:cacc1} follows from \cite[Proposition 8]{BDGN}. To show \eqref{eq:cacc2}, let $\eta$ be a smooth cut-off function with $\chi_{B_r}\leq\eta\leq\chi_{2B_r}$ and $|\nabla \eta|\leq\frac{c}{r}$. Testing $\eta^p u\in W^{1,p}_{0,\omega}(2\Omega_r)$ in \eqref{eq:plap4}, we get

\begin{align*}
\int_{2\Omega_{r}}|\setM\nabla u|^{p-2}\setM\nabla u\cdot\setM\nabla(\eta^p u)\,dx=\int_{2\Omega_{r}}|\setM F|^{p-2}\setM F\cdot\setM\nabla(\eta^p u)\,dx.
\end{align*}
Using \eqref{eq:M2} we have
\begin{align*}
\int_{2\Omega_{r}}\eta^{p}|\nabla u|^p\omega^p\,dx&\lesssim \int_{2\Omega_{r}}\eta^{p-1}|\nabla u|^{p-1}\left|\dfrac{u}{r}\right|\omega^p\,dx\\
&\quad+\int_{2\Omega_{r}}\eta^p|F|^{p-1}|\nabla u|\omega^p\,dx+\int_{2\Omega_{r}}\eta^{p-1}|F|^{p-1}\left|\dfrac{u}{r}\right|\omega^p\,dx.
\end{align*}
By Young's inequality, absorb the term $\eta^p|\nabla u|^p\omega^p$ into left-hand side, it follows that
\begin{align*}
\int_{\Omega_{r}}|\nabla u|^p\omega^p\,dx\lesssim\int_{2\Omega_{r}}\left|\dfrac{u}{r}\right|^p\omega^p\,dx+\int_{2\Omega_{r}}|F|^p\omega^p\,dx.
\end{align*}
Now, by \eqref{eq:m.den1}, we have $|\Omega_r|\eqsim|B_r|\eqsim|2B_{r}|\eqsim|2\Omega_{r}|$. Then \eqref{eq:cacc2} follows.
\end{proof}

Now, we can provide the reverse H\"{o}lder's inequality.
\begin{lemma}\label{lem:rev.holder0}
  Assume \eqref{eq:Omega}. There exists $\kappa_2=\kappa_2(n,p,\Lambda)>0$ and $\theta\in(0,1)$ such that if $\lvert\log\setM|_{\setBMO(3B_r)}\leq\kappa_2$, then
  \begin{align*}
    \dashint_{\Omega_{r}}|\nabla u|^p\omega^p\,dx\leq c\left(\dashint_{2\Omega_{r}}(|\nabla u|\omega)^{\theta p}\,dx\right)^{\frac{1}{\theta}}+c\dashint_{2\Omega_{r}}|F|^p\omega^p\,dx
  \end{align*}
  holds with $c=c(n,p,\Lambda)$.
\end{lemma}
\begin{proof}
  If $\frac{3}{2}B_r\subset\Omega$, then we can select small $\kappa_2$  such that Lemma \ref{lem:Mu.weight} (c) holds true to use Proposition 3 in \cite{BDGN}. This and Proposition \ref{prop:cacc} proves the claim.

  If $\frac{3}{2}B_r\not\subset\Omega$, using Proposition \ref{prop:SP} instead of Proposition 3 in \cite{BDGN}, we again prove the claim.
\end{proof}

Now, we have the following higher integrability.

\begin{corollary}[Higher integrability]\label{cor:hi} 
  Assume \eqref{eq:Omega}. There exist $\kappa_2=\kappa_2(n,p,\Lambda)>0$ and $s=s(n,p,\Lambda)\in(1,2)$ such that if $\lvert\log\setM|_{\setBMO(3B_r)}\leq\kappa_2$, then
  \begin{align*}
    \left(\dashint_{\Omega_{r}}(|\nabla u|^p\omega^p)^s\,dx\right)^{\frac{1}{s}}\leq c\dashint_{2\Omega_{r}}|\nabla u|^p\omega^p\,dx+c\left(\dashint_{2\Omega_{r}}(|F|^p\omega^p)^s\,dx\right)^{\frac{1}{s}}
  \end{align*}
  holds with $c=c(n,p,\Lambda)$.
\end{corollary}
\begin{proof}
  After extending $\nabla u=F=0$ in $2B_{r}\setminus\Omega$, and considering $|\Omega_r|\eqsim|B_r|$ and $|2\Omega_{r}|\eqsim|2B_{r}|$, Gehring's lemma (e.g. \cite[Theorem 6.6]{G2}) implies the conclusion.
\end{proof}

\subsection{Comparison Estimate at Boundary}
In this section, we only prove boundary comparison estimates, since the interior estimates are proved in \cite{BDGN}. 
Let us assume $\setM F\in L^{q}(\Omega)$. Choose a cut-off function $\eta\in C^{\infty}_{0}(B_0)$ with
\begin{align}\label{eq:cutoff1}
\chi_{\frac{1}{2}B_0}\leq\eta\leq\chi_{B_0}\quad\text{and}\quad \|\nabla \eta\|_{\infty}\leq c/R.
\end{align}
Define $z$ on $\setR^n$ as follows: first let $z$ on $B_0\cap\Omega$ be such that
\begin{align}\label{eq:z}
z:=u\eta^{p'}
\end{align}
and take the zero extension for $z$ on $\setR^n\setminus (B_0\cap\Omega)$, if necessary. Also, we denote
\begin{align}\label{eq:g1}
g:=\eta^{p'}\nabla u-\nabla z=-u\nabla(\eta^{p'})=-up'\eta^{p'-1}\nabla \eta.
\end{align}
Then we have the following estimate:
\begin{align}\label{eq:g2}
|g|\lesssim\dfrac{|u|}{R}.
\end{align}

For the convenience of notation, we write
\begin{align*}
\setM_{B_r}&:=\mean{\setM}^{\log}_{B_r},\\\omega_{B_r}&:=\mean{\omega}^{\log}_{B_r},\\
\end{align*}
and so
\begin{align*}
\mathcal{A}_{B_r}(\xi)&:=|\setM_{B_r}\xi|^{p-2}\setM^{2}_{B_r}\xi=\setM_{B_r}A(\setM_{B_r}\xi),\\
\mathcal{V}(x,\xi)&:=V(\setM(x)\xi),\\
\mathcal{V}_{B_r}(\xi)&:=V(\setM_{B_r}\xi).
\end{align*}
Then we have the following relations for all $\xi\in\setR^n$:
\begin{align}
\mathcal{A}(x,\xi)\cdot\xi&=|\mathcal{V}(x,\xi)|^2,\label{eq:calA1}\\\mathcal{A}_{B_r}(\xi)\cdot\xi&=|\mathcal{V}_{B_r}(\xi)|^2,\label{eq:calA2}\\
|\mathcal{A}_{B_r}(\xi)|&\lesssim\omega^{p}_{B_r}|\xi|^{p-1} \label{eq:calA3}
\end{align}
and by \cite[Section 3]{BDGN},
\begin{align}
\Lambda^{-1}\omega_{B_r}|\xi|&\leq|\setM_{B_r}\xi|\leq\omega_{B_r}|\xi|,\label{eq:MB2}\\
\Lambda^{-1}\omega_{B_r}&\leq|\setM_{B_r}|\leq\omega_{B_r}\label{eq:MB3}.
\end{align}
Summing up the above result, we have \cite[Lemma 16]{BDGN} as follows:
for all $\xi\in\setR^n$ and all $x\in B_r$ there holds
\begin{align}\label{eq:diff.A}
|\mathcal{A}_{B_r}(\xi)-\mathcal{A}(x,\xi)|\lesssim\dfrac{|\setM-\setM_{B_r}(x)|}{|\setM_{B_r}|}\left(|\mathcal{A}_{B_r}(\xi)|+|\mathcal{A}(x,\xi)|\right).
\end{align}

Before introducing the reference problem, we provide the following lemma for the well-posedness:
\begin{lemma}\label{lem:HI1}
Assuming \eqref{eq:Omega}, there exists $\kappa_3=\kappa_3(n,p,\Lambda)>0$ and $s=s(n,p,\Lambda)\in(1,2)$ such that if $\lvert\log\setM|_{\setBMO(4B_r)}\leq\kappa_3$, then
with $2\Omega_{r}=2B_r\cap B_0\cap\Omega$ there holds
\begin{align}\label{eq:HI1}
\dashint_{\Omega_r}(|\nabla u|\omega_{B_r})^p\,dx\lesssim\dashint_{2\Omega_{r}}(|\nabla u|\omega)^p\,dx+\left(\dashint_{2\Omega_{r}}(|F|\omega)^{ps}\,dx\right)^{\frac{1}{s}}.
\end{align}
\end{lemma}
\begin{proof}
Using H\"{o}lder's inequality, $|\Omega_r|\eqsim|B_r|$ and Lemma \ref{lem:Mu.weight}, we have
\begin{align*}
\dashint_{\Omega_r}(|\nabla u|\omega_{B_r})^p\,dx&\leq\left(\dashint_{\Omega_r}(|\nabla u|\omega)^{ps}\,dx\right)^{\frac{1}{s}}\left(\dashint_{\Omega_r}(\omega_{B_r}\omega^{-1})^{ps'}\,dx\right)^{\frac{1}{s'}}\\
&\lesssim\left(\dashint_{\Omega_r}(|\nabla u|\omega)^{ps}\,dx\right)^{\frac{1}{s}}\left(\dashint_{B_r}(\omega_{B_r}\omega^{-1})^{ps'}\,dx\right)^{\frac{1}{s'}}\\
&\lesssim\left(\dashint_{\Omega_r}(|\nabla u|\omega)^{ps}\,dx\right)^{\frac{1}{s}}.
\end{align*}
Then Corollary \ref{cor:hi} yields the conclusion.
\end{proof}

Now, let $h\in W^{1,p}_{\omega_{B_r}}(\Omega_r)$ be the weak solution of
\begin{align}\label{eq:2nd}
\begin{split}
-\divergence\left(\mathcal{A}_{B_r}(\nabla h)\right)&=0\quad\text{in }\Omega_r,\\
h&=z\quad\text{on }\partial\Omega_r.
\end{split}
\end{align}
Then $h$ is the unique minimizer of
\begin{align}\label{eq:mini}
w\mapsto\int_{\Omega_r}\phi(|\setM_{B_r}\nabla w|)\,dx
\end{align}
with boundary data $w=z$ on $\partial\Omega_r$. Now, we provide the first comparison estimate. Recall that $B_r=B_r(\tilde{x})$, $B_0=B_R(x_0)$, $4B_r\subset 2B_0$, $\tilde{x}\in\overline{\Omega}$ and $z,h$ are given by \eqref{eq:z} and \eqref{eq:2nd}, respectively. Moreover, as \cite[Eq. (3.25)]{BDGN}, we have
\begin{align}\label{eq:eq}
\begin{split}
-&\divergence\left(\mathcal{A}_{B_r}(\nabla z)-\mathcal{A}_{B_r}(\nabla h)\right)\\
&\quad=-\divergence\left(\mathcal{A}_{B_r}(\nabla z)-\mathcal{A}(\cdot,\nabla z)\right)-\divergence\left(\mathcal{A}(\cdot,\nabla z)-\mathcal{A}(\cdot,\nabla z+g)\right)\\
&\quad\quad-\eta^p\divergence\left(\mathcal{A}(\cdot,F)\right)-\nabla(\eta^p)\cdot\mathcal{A}(\cdot,\nabla u)
\end{split}
\end{align}
on $\Omega_r$, in the distributional sense.

\begin{proposition}[{First comparison at boundary}]\label{prop:comp}
Assuming \eqref{eq:Omega}, let $h$ be as in \eqref{eq:2nd} and $z$ be as in \eqref{eq:z}. There exist $s>1$ and $\kappa_4=\kappa_4(n,p,\Lambda)\in(0,1)$, such that if $\lvert\log\setM|_{\setBMO(4B_r)}\leq\kappa_4$ holds, then for every $\epsilon\in(0,1)$ we have
\begin{align}\label{eq:comp}
\begin{split}
\dashint_{\Omega_r}&|\mathcal{V}_{B_r}(\nabla h)-\mathcal{V}_{B_r}(\nabla z)|^2\,dx\leq c\left(\lvert\log\setM|^2_{\setBMO(B_r)}+\epsilon\right)\left(\dashint_{\Omega_r}(|\nabla z|^p\omega^p)^s\,dx\right)^{\frac{1}{s}}\\
&\quad\quad\quad+c\,C^*(\epsilon)\left(\dashint_{2\Omega_{r}}\left(\dfrac{|u|^p}{R^p}\omega^p\right)^s\,dx\right)^{\frac{1}{s}}+c\,C^*(\epsilon)\left(\dashint_{2\Omega_{r}}(|F|^p\omega^p)^s\,dx\right)^{\frac{1}{s}}
\end{split}
\end{align}
for some $c=c(n,p,\Lambda)$, and $C^*(\epsilon)=\max\left\{\epsilon^{1-p},\epsilon^{-\frac{1}{p-1}}\right\}$.
\end{proposition}

\begin{proof}
The proof is similar to the one of \cite[Proposition 17]{BDGN}. Observe that $\lvert\log\setM|_{\setBMO(4B_r)}\leq\kappa_4$ implies that \eqref{eq:2nd} is well-defined. Testing $z-h$ to \eqref{eq:2nd} and \eqref{eq:plap4}, by \eqref{eq:eq} it follows that
\begin{align}\label{eq:comp1}
\begin{split}
I_0&:=\dashint_{\Omega_r}\left(\mathcal{A}_{B_r}(\nabla z)-\mathcal{A}_{B_r}(\nabla h)\right)\cdot(\nabla z-\nabla h)\,dx\\
&=\dashint_{\Omega_r}\left(\mathcal{A}_{B_r}(\nabla z)-\mathcal{A}(x,\nabla z)\right)\cdot(\nabla z-\nabla h)\,dx\\
&\quad+\dashint_{\Omega_r}\left(\mathcal{A}(x,\nabla z)-\mathcal{A}(x,\nabla z+g)\right)\cdot(\nabla z-\nabla h)\,dx\\
&\quad+\dashint_{\Omega_r}\mathcal{A}(x,F)\cdot(\nabla(\eta^p z)-\nabla(\eta^p h))\,dx\\
&\quad+\dashint_{\Omega_r}\nabla (\eta^p)\cdot\mathcal{A}(x,\nabla u)(z-h)\,dx=:I_1+I_2+I_3+I_4.
\end{split}
\end{align}

By Lemma \ref{lem:equivA}, we have
\begin{align}\label{eq:comp2}
I_0\eqsim\dashint_{\Omega_r}|\mathcal{V}_{B_r}(\nabla h)-\mathcal{V}_{B_r}(\nabla z)|^2\,dx\eqsim\dashint_{\Omega_r}\phi_{|\nabla z|}(|\nabla z-\nabla h|)\omega^p_{B_r}\,dx.
\end{align}	
To estimate $I_1$, arguing as in the proof of \cite[Proposition 17]{BDGN}, we have
\begin{align}\label{eq:comp3}
\begin{split}
I_1&=\dashint_{\Omega_r}\left(\mathcal{A}_{B_r}(\nabla z)-\mathcal{A}(x,\nabla z)\right)\cdot(\nabla z-\nabla h)\,dx\\
&\leq\sigma\dashint_{\Omega_r}\phi_{|\nabla z|}(|\nabla z-\nabla h|)\omega^p_{B_r}\,dx\\
&\quad+c(\sigma)\dashint_{\Omega_r}\left(\dfrac{|\setM-\setM_{B_r}|}{|\setM_{B_r}|}\right)^2\phi(|\nabla z|)\left(\dfrac{\omega^p_{B_r}}{\omega^p}+\dfrac{\omega^p}{\omega^p_{B_r}}+\dfrac{\omega^{p'}}{\omega^{p'}_{B_r}}\right)\omega^p\,dx=I_{1,1}+I_{1,2}
\end{split}
\end{align}
for any $\sigma\in(0,1)$. Now, $I_{1,1}$ is absorbed to the left-hand side $I_0$ by choosing $\sigma=\sigma(n,p,\Lambda)$ sufficiently small, and so $c(\sigma)=c$. For $I_{1,2}$, we first assume $\lvert\log\setM|_{\setBMO(B_r)}\leq\kappa_1=\kappa_1(n,p,\Lambda)$ and then use Lemma \ref{lem:BMOM}, together with $|\Omega_r|\eqsim |B_r|$ (the measure density of $\Omega_r$ to $B_r$) from \eqref{eq:Omega} and \eqref{eq:m.den1}, to have
\begin{align}\label{eq:comp3.1}
\begin{split}
I_{1,3}&:=\left(\dashint_{\Omega_r}\left(\dfrac{|\setM-\setM_{B_r}|}{|\setM_{B_r}|}\right)^{4s'}\,dx\right)^{\frac{1}{2s'}}\\
&\lesssim \left(\dashint_{B_r}\left(\dfrac{|\setM-\setM_{B_r}|}{|\setM_{B_r}|}\right)^{4s'}\,dx\right)^{\frac{1}{2s'}}\lesssim \lvert\log\setM|^2_{\setBMO(B_r)}.
\end{split}
\end{align}
Also, assume $\lvert\log\setM|_{\setBMO(B_r)}\leq\kappa_4$ for some small $\kappa_4=\kappa_4(n,p,\Lambda)$ and then use Lemma \ref{lem:Mu.weight}, together with $|\Omega_r|\eqsim |B_r|$ from \eqref{eq:m.den1} with the help of \eqref{eq:Omega} to have 
\begin{align}\label{eq:comp3.2}
I_{1,4}:=\dashint_{\Omega_r}\left(\dfrac{\omega^p_{B_r}}{\omega^p}+\frac{\omega^p}{\omega^p_{B_r}}+\frac{\omega^{p'}}{\omega^{p'}_{B_r}}\right)^{2s'}\,dx\lesssim\dashint_{B_r}\left(\dfrac{\omega^p_{B_r}}{\omega^p}+\frac{\omega^p}{\omega^p_{B_r}}+\frac{\omega^{p'}}{\omega^{p'}_{B_r}}\right)^{2s'}\,dx\leq c.
\end{align}
Then by H\"{o}lder's inequality and the above two displays, we obtain
\begin{align}\label{eq:comp4}
I_{1,2}&\leq c\, I_{1,3}I_{1,4}^{\frac{1}{2s'}}\left(\dashint_{\Omega_r}(|\nabla z|^p\omega^p)^s\,dx\right)^{\frac{1}{s}}\lesssim c\,\lvert\log\setM|^2_{\setBMO(B_r)}\left(\dashint_{\Omega_r}(|\nabla z|^p\omega^p)^s\,dx\right)^{\frac{1}{s}}.
\end{align}

From now on, we only specify the necessary tools to provide each resulting estimates, since we mainly follow the proof of \cite[Proposition 17]{BDGN}, and when we use the measure density property, we apply the similar manipulation as above. First, by \eqref{eq:g2}, Lemma \ref{lem:change} and Lemma \ref{lem:Mu.weight} together with the measure density property of $\Omega_r$ to $B_r$, we estimate $I_2$ as follows:
\begin{align}\label{eq:comp5}
\begin{split}
I_{2}&\leq\sigma\dashint_{\Omega_r}\phi_{|\nabla z|}(|\nabla z-\nabla h|)\omega^p_{B_r}\,dx\\
&\quad+ c_{\sigma}\epsilon\left(\dashint_{\Omega_r}(|\nabla z|^p\omega^p)^s\,dx\right)^{\frac{1}{s}}+c_{\sigma}\epsilon^{-\frac{1}{p-1}}\left(\dashint_{\Omega_r}\left(\dfrac{|u|^p}{ R^p}\omega^p\right)^s\,dx\right)^{\frac{1}{s}}
\end{split}
\end{align}
for any $\sigma,\epsilon\in(0,1)$. To estimate $I_3$, using Young's inequality and $0\leq\eta\leq 1$, we have
\begin{align}\label{eq:comp6}
\begin{split}
I_3&\lesssim \epsilon^{-\frac{1}{p-1}}\dashint_{\Omega_r}\dfrac{\omega^{p'}}{\omega^{p'}_{B_r}}\eta^p|F|^p\omega^p\,dx+\epsilon\dashint_{\Omega_r}|\nabla z-\nabla h|^p\omega^p_{B_r}\,dx+\epsilon\dashint_{\Omega_r}\left|\dfrac{z-h}{r}\right|^p\omega^p_{B_r}\,dx\\
&=:I_{3,1}+I_{3,2}+I_{3,3}.
\end{split}
\end{align}
Extending $z-h$ as $0$ in $B_r\setminus\Omega$, and using Proposition 3 in \cite{BDGN}, we have $I_{3,3}\lesssim I_{3,2}$. We employ triangle inequality, minimizing property of $h$ in \eqref{eq:mini}, \eqref{eq:MB2}, H\"{o}lder's inequality and Lemma \ref{lem:Mu.weight} together with $|\Omega_r|\eqsim |B_r|$, to obtain
\begin{align}\label{eq:comp8}
I_{3,2}\leq\epsilon c\left(\dashint_{\Omega_r}(|\nabla z|^p\omega^p)^s\,dx\right)^{\frac{1}{s}}.
\end{align}
Using also H\"{o}lder's inequality and Lemma \ref{lem:Mu.weight} with $|\Omega_r|\eqsim |B_r|$, we have
\begin{align}\label{eq:comp9}
I_{3,1}\lesssim c\epsilon^{-\frac{1}{p-1}} \left(\dashint_{\Omega_r}(|F|^p\omega^p)^s\,dx\right)^{\frac{1}{s}}.
\end{align}

Finally, to estimate $I_4$, instead of dividing the case into $p>2$ and $1<p\leq2$ as in \cite{BDGN}, we consider the cases in a unified way. By $p'(p-2+\frac{1}{p})=p-1$ and $\eta^{p'}\nabla u=\nabla z+g$, we first see that
\begin{align}\label{eq:z8}
\begin{split}
I_4&\leq\dashint_{\Omega_r}|\nabla(\eta^p)||\mathcal{A}(x,\nabla u)||z-h|\,dx\\
&\lesssim\dashint_{\Omega_r}|\nabla \eta||\eta^{p'}\nabla u|^{p-2+\frac{1}{p}}|\nabla u|^{1-\frac{1}{p}}|z-h|\omega^p\,dx\\
&\lesssim\dashint_{\Omega_r}\dfrac{r}{R}|\nabla z+g|^{p-2+\frac{1}{p}}|\nabla u|^{1-\frac{1}{p}}\left|\dfrac{z-h}{r}\right|\omega^p\,dx\\
&\lesssim\epsilon\dashint_{\Omega_r}|\nabla z+g|^p\omega^p\,dx+C^*(\epsilon)\dashint_{\Omega_r}\left(\dfrac{r}{R}\right)^{\frac{p^2}{p-1}}|\nabla u|^p\omega^p\,dx+\epsilon\dashint_{\Omega_r}\left|\dfrac{z-h}{r}\right|^p\omega^p\,dx
\end{split}
\end{align}
for any $\epsilon\in(0,1]$, where for the last step we have used Young's inequality for the exponents $\left(\frac{p}{p-2+\frac{1}{p}},\frac{p^2}{p-1},p\right)$ and $0\leq\eta_1\leq 1$. Here, $C^*(t):(0,1]\rightarrow\setR^+$ is such that
\begin{align}\label{eq:C}
C^*(t)=\max\left\{t^{1-p},t^{-\frac{1}{p-1}}\right\},
\end{align}
which is a continuous function on $(0,1]$ for each fixed $p\in(1,\infty)$. Recalling \eqref{eq:g2}, we have
\begin{align}\label{eq:z9}
\dashint_{\Omega_r}|\nabla z+g|^p\omega^p\,dx\lesssim\dashint_{\Omega_r}|\nabla z|^p\omega^p\,dx+\dashint_{\Omega_r}\left|\dfrac{u}{R}\right|^p\omega^p\,dx.
\end{align}
Also, since $\frac{p^2}{p-1}>p$ on $p\in(1,\infty)$ and $r\leq R$ hold, there holds
\begin{align}\label{eq:z10}
\begin{split}
\dashint_{\Omega_r}\left(\dfrac{r}{R}\right)^{\frac{p^2}{p-1}}|\nabla u|^p\omega^p\,dx&\lesssim\left(\dfrac{r}{R}\right)^{p}\dashint_{\Omega_r}|\nabla u|^p\omega^p\,dx\\
&\lesssim \left(\dfrac{r}{R}\right)^{p}\dashint_{2\Omega_r}\left|\dfrac{u}{r}\right|^p\omega^p\,dx+\dashint_{2\Omega_r}|F|^p\omega^p\,dx\\
&\lesssim \dashint_{2\Omega_r}\left|\dfrac{u}{R}\right|^p\omega^p\,dx+\dashint_{2\Omega_r}|F|^p\omega^p\,dx.
\end{split}
\end{align}

Extending $z-h$ as $0$ in $B_r\setminus\Omega$, and using Proposition 3 in \cite{BDGN}, H\"{o}lder's inequality, Lemma \ref{lem:Mu.weight} together with $|\Omega_r|\eqsim |B_r|$, triangle inequality, minimizing property of $h$ in \eqref{eq:mini} and \eqref{eq:MB2} yields 
\begin{align}\label{eq:comp13}
\dashint_{\Omega_r}\left|\dfrac{z-h}{r}\right|^p\omega^p\,dx\lesssim\left(\dashint_{\Omega_r}(|\nabla z|^p\omega^p)^s\,dx\right)^{\frac{1}{s}},
\end{align}
as in \cite{BDGN}. Note that the argument used in \cite{BDGN} for \eqref{eq:comp13} can be applied in all cases $p\in(1,\infty)$.
Thus it follows that
\begin{align}\label{eq:z11}
\begin{split}
I_3\lesssim \epsilon\left(\dashint_{\Omega_r}(|\nabla z|^p\omega^p)^s\,dx\right)^{\frac{1}{s}}+C^*(\epsilon)\dashint_{2\Omega_r}\left|\dfrac{u}{R}\right|^p\omega^p\,dx+C^*(\epsilon)\dashint_{2\Omega_r}|F|^p\omega^p\,dx.
\end{split}
\end{align} 
Summing up the above all estimates, we have
\begin{align}\label{eq:comp14}
\begin{split}
  \dashint_{\Omega_r}|\mathcal{V}_{B_r}(\nabla h)-\mathcal{V}_{B_r}(\nabla z)|^2\,dx&\lesssim\sigma\dashint_{\Omega_r}\phi_{|\nabla z|}(|\nabla z-\nabla h|)\omega^p_{B_r}\,dx\\
  &\quad+c(\sigma)(\,\lvert\log\setM|^2_{\setBMO(B_r)}+\epsilon)\left(\dashint_{\Omega_r}(|\nabla z|^p\omega^p)^s\,dx\right)^{\frac{1}{s}}\\
  &\quad+c(\sigma)\big(\epsilon^{-\frac{1}{p-1}}+C^*(\epsilon)\big)\left(\dashint_{2\Omega_r}\left(\dfrac{|u|^p}{R^p}\omega^p\right)^s\,dx\right)^{\frac{1}{s}}\\
  &\quad+c(\sigma)\big(\epsilon^{-\frac{1}{p-1}}+C^*(\epsilon)\big) \left(\dashint_{2\Omega_r}(|F|^p\omega^p)^s\,dx\right)^{\frac{1}{s}}\\
  &=I_5+I_6+I_7+I_8.
\end{split}
\end{align}
By Lemma \ref{lem:equivA} and \eqref{eq:M2}, $\phi_{|\nabla z|}(|\nabla z-\nabla h|)\omega^p_{B_r}\eqsim |\mathcal{V}_{B_r}(\nabla h)-\mathcal{V}_{B_r}(\nabla z)|^2$. Then by choosing $\sigma\in(0,1)$ sufficiently small depending on $n,p$ and $\Lambda$, $I_5$ is absorbed to the left-hand side. Finally, $\epsilon^{-\frac{1}{p-1}}\leq C^*(\epsilon)$ holds when $\epsilon\in(0,1)$, and so the estimate \eqref{eq:comp} holds true.
\end{proof}

Now, we give the second comparison estimate. With $y=\Psi(x)$, we define
\begin{align}\label{eq:trans1}
\tilde{h}(y):=h(\Psi^{-1}(y)).
\end{align}
Let $\tilde{v}=\tilde{v}(y)\in W^{1,p}(\Psi(\frac{1}{2}\Omega_{r}))$ be the weak solution of 
\begin{align}\label{eq:2comp3}
\begin{split}
-\divergence_y(|\setM_{B_r}\nabla_y \tilde{v}|^{p-2}\setM_{B_r}^2\nabla_y \tilde{v})&=0\quad\text{in }\Psi(\tfrac{1}{2}\Omega_{r}),\\
\tilde{v}&=\tilde{h}\quad\text{on }\partial(\Psi(\tfrac{1}{2}\Omega_{r})).
\end{split}
\end{align}
Since $\Psi$ is a homeomorphism, together with \eqref{eq:lip1} we have $\partial(\Psi(\frac{1}{2}\Omega_{r}))=\Psi(\partial(\frac{1}{2}\Omega_{r}))$. Denoting $v(x)=\tilde{v}(y)=\tilde{v}(\Psi(x))$, we have
\begin{align}\label{eq:trans2}
\nabla v(x)=\nabla\tilde{v}(\Psi(x))=(\nabla\Psi)(x)\nabla_y\tilde{v}(\Psi(x))=(\nabla\Psi)(x)\nabla_y\tilde{v}(y).
\end{align}
 Denoting $T(x)=(\nabla\Psi)^{-1}(x)$, \eqref{eq:2comp3} transforms into
\begin{align}\label{eq:2comp1}
\begin{split}
-\divergence(T^{t}(x)|\setM_{B_r} T(x)\nabla&v|^{p-2}\setM_{B_r}^2T(x)\nabla v)=0\quad\text{in }\tfrac{1}{2}\Omega_{r},\\
v(x)&=h(x)\quad\quad\quad\quad\quad\quad\quad\,\text{on }\partial
(\tfrac{1}{2}\Omega_{r}),
\end{split}
\end{align}
where~$T^t(x)$  abbreviation stands for transpose. Or equivalently, denoting 
\begin{align}\label{eq:A.Psi}
\mathcal{A}_{\Psi}(x,\xi):=T^{t}(x)|\setM_{B_r}T(x)\xi|^{p-2}\setM_{B_r}^2T(x)\xi,
\end{align}
we have
\begin{align}\label{eq:2comp2}
\begin{split}
-\divergence(\mathcal{A}_{\Psi}(x,\nabla v))&=0\quad\quad\,\,\text{in }\tfrac{1}{2}\Omega_{r},\\
v(x)&=h(x)\quad\text{on }\partial(\tfrac{1}{2}\Omega_{r}).
\end{split}
\end{align}
The problems \eqref{eq:2comp1} and \eqref{eq:2comp2} can be derived also from the weak formulation of the equation. At this time we have also used that $\det(\nabla\Psi)=1$ for the change of coordinate in the integrals.
The natural function space for $v$ is $W^{1,p}_{\omega_{B_r}}(\frac{1}{2}\Omega_r)$ and $v$ is the unique minimizer of
\begin{align}\label{eq:mini2}
w\mapsto\int_{\frac{1}{2}\Omega_{r}}\phi(|\setM_{B_r}T(x)\nabla w|)\,dx
\end{align}
subject to the boundary condition $v=h$ on $\partial(\frac{1}{2}\Omega_{r})$.

Now, we need the following lemma.
\begin{lemma}\label{lem:diff.A2}
Assume \eqref{eq:Omega}. For all $\xi\in\setR^n$ and $x\in \frac{1}{2}\Omega_r$ we have
\begin{align}\label{eq:T}
|T(x)|\eqsim|T^{-1}(x)|\eqsim c
\end{align}
and
\begin{align}\label{eq:diff.A2}
|\mathcal{A}_{B}(\xi)-\mathcal{A}_{\Psi}(x,\xi)|\leq c\|\nabla\psi\|_{\infty}\omega_B^p\min\left\{|\xi|^{p-1},|T(x)\xi|^{p-1}\right\}
\end{align}
for some $c=c(n,p,\Lambda)$.
\end{lemma}
\begin{proof}
First, since $(T(x)-\identity)^2=((\nabla\Psi(x))^{-1}-\identity)^2=0$, we have $T^{-1}(x)-\identity=\identity-T(x)$. Then
\begin{align}\label{eq:diff1}
|T^{-1}(x)-\identity|=|\identity-T(x)|\leq n\|\nabla\psi\|_{\infty}
\end{align}
holds, and so together with $\|\nabla\psi\|_{\infty}\leq\frac{1}{2n}$, it follows that
\begin{align}\label{eq:diff2}
\frac{1}{2}\leq|\identity|-n\|\nabla\psi\|_{\infty}\leq|T(x)|\leq|\identity|+n\|\nabla\psi\|_{\infty}\leq\frac{3}{2}.
\end{align}
Hence we have $|T(x)|\eqsim c$. Similarly, we have $|T^{-1}(x)|\eqsim c$.

On the other hand, observe that by \eqref{eq:MB2},
\begin{align*}
|\mathcal{A}_B(\xi)-\mathcal{A}_{\Psi}(x,\xi)|
&=\left||\setM_B\xi|^{p-2}\setM^2_{B}\xi-|\setM_B T(x)\xi|^{p-2}T^{t}(x)\setM^2_{B}T(x)\xi\right|\\
&\leq\left|\setM_B|\setM_B\xi|^{p-2}\setM_B\xi-\setM_B|\setM_BT(x)\xi|^{p-2}\setM_BT(x)\xi\right|\\
&\quad+|\setM_B-T^{t}(x)\setM_B|\cdot\left||\setM_BT(x)\xi|^{p-2}\setM_BT(x)\xi\right|\\
&\lesssim\omega_B\cdot|A(\setM_B\xi)-A(\setM_BT(x)\xi)|+\omega_B\cdot|\identity-T^{t}(x)|\cdot|\setM_BT(x)\xi|^{p-1}.
\end{align*}
Here, by \eqref{eq:T} and \eqref{eq:MB2},
\begin{align*}
|\setM_{B_r}\xi|=|\setM_{B_r}T^{-1}(x)T(x)\xi|\lesssim|\setM_{B_r}|\cdot|T^{-1}(x)|\cdot|T(x)\xi|\lesssim|\setM_{B_r}|\cdot|T(x)\xi|\lesssim|\setM_{B_r}T(x)\xi|
\end{align*}
and similarly $|\setM_{B_r}T(x)\xi|\lesssim |\setM_{B_r}T(x)T^{-1}(x)\xi|= |\setM_{B_r}\xi|$ holds. Thus we have $|\setM_{B_r}T(x)\xi|\eqsim|\setM_{B_r}\xi|$. Then together with \eqref{eq:diff1} and Lemma \ref{lem:equivA}, there holds
\begin{align*}
|A(\setM_{B_r}\xi)&-A(\setM_{B_r}T(x)\xi)|\eqsim\phi'_{|\setM_{B_r}\xi|\vee|\setM_{B_r}T(x)\xi|}\left(|\setM_{B_r}\xi-\setM_{B_r}T(x)\xi|\right)\\
&\eqsim\left(|\setM_{B_r}\xi|\vee|\setM_{B_r}T(x)\xi|\vee|\setM_{B_r}\xi-\setM_{B_r}T(x)\xi|\right)^{p-2}|\setM_{B_r}\xi-\setM_{B_r}T(x)\xi|\\
&\lesssim\omega_{B_r}|\identity-T(x)||\xi|\left(|\setM_{B_r}\xi|\vee|\setM_{B_r}T(x)\xi|\vee|\setM_{B_r}\xi-\setM_{B_r}T(x)\xi|\right)^{p-2}\\
&\lesssim\omega_{B_r}|\identity-T(x)||\xi|\dfrac{\left(|\setM_{B_r}\xi|+|\setM_{B_r}T(x)\xi|\right)^{p-1}}{|\setM_{B_r}\xi|}\\
&\lesssim\Lambda|\identity-T(x)|\left(|\setM_{B_r}\xi|+|\setM_{B_r}T(x)\xi|\right)^{p-1}\\
&\lesssim\|\nabla\psi\|_{\infty}|\setM_{B_r}T(x)\xi|^{p-1}.
\end{align*} 
Thus, together with $|\identity-T^t(x)|=|(\identity-T(x))^t|=|\identity-T(x)|\leq n\|\nabla\psi\|_{\infty}$, we have
\begin{align*}
|\mathcal{A}_{B_r}(\xi)-\mathcal{A}_{\Psi}(x,\xi)|\lesssim\omega_{B_r}\|\nabla\psi\|_{\infty}|\setM_{B_r}T(x)\xi|^{p-1}\lesssim\omega_{B_r}^p\|\nabla\psi\|_{\infty}|T(x)\xi|^{p-1}.
\end{align*}
Finally, since $\omega_{B_r}|T(x)\xi|\eqsim|\setM_{B_r}T(x)\xi|\eqsim|\setM_{B_r}\xi|\eqsim\omega_{B_r}|\xi|$, we get the conclusion.
\end{proof}

Now, we can compute the comparison estimate.

\begin{proposition}[{Second comparison at boundary}]\label{prop:comp2}
	Assuming \eqref{eq:Omega}, let $v$ be as in \eqref{eq:2comp1} and $h$ be as in \eqref{eq:2nd}. There exists $\kappa_4=\kappa_4(n,p,\Lambda)$ such that if $\lvert\log\setM|_{\setBMO(4B_r)}\leq\kappa_4$, then for some $c=c(n,p,\Lambda)$, we have 
	\begin{align}\label{eq:2comp}
	\begin{split}
	\dashint_{\frac{1}{2}\Omega_{r}}&|\mathcal{V}_{B_r}(\nabla v)-\mathcal{V}_{B_r}(\nabla h)|^2\,dx\leq c\|\nabla\psi\|^2_{\infty}\dashint_{\Omega_{r}}(|\nabla h|^p\omega_{B_r}^p)\,dx.
	\end{split}
	\end{align}
\end{proposition}
\begin{proof}
Test $v-h$ to \eqref{eq:2nd} and \eqref{eq:2comp1} to have
\begin{align}\label{eq:2comp4}
\begin{split}
\dashint_{\frac{1}{2}\Omega_{r}}&\left(\mathcal{A}_{B_r}(\nabla v)-\mathcal{A}_{B_r}(\nabla h)\right)\cdot(\nabla v-\nabla h)\,dx\\
&=\dashint_{\frac{1}{2}\Omega_{r}}\left(\mathcal{A}_{B_r}(\nabla v)-\mathcal{A}_{\Psi}(x,\nabla v)\right)\cdot(\nabla  v-\nabla h)\,dx.
\end{split}
\end{align}
We apply Lemma \ref{lem:equivA}, \eqref{eq:diff.A2} and then use Young's inequality to obtain
\begin{align}\label{eq:2comp5}
\begin{split}
\dashint_{\frac{1}{2}\Omega_{r}}|\mathcal{V}_{B_r}(\nabla v)-\mathcal{V}_{B_r}(\nabla  h)|^2\,dx&\eqsim\dashint_{\frac{1}{2}\Omega_{r}}\phi_{|\nabla v|}(|\nabla  v-\nabla h|)\omega^p_{B_r}\,dx\\
&\lesssim\dashint_{\frac{1}{2}\Omega_{r}}|\mathcal{A}_{\Psi}(x,\nabla v)-\mathcal{A}_{B_r}(\nabla v)||\nabla v-\nabla h|\,dx\\
&\lesssim\dashint_{\frac{1}{2}\Omega_{r}}\|\nabla\psi\|_{\infty}\omega_{B_r}^p\phi'(|\nabla v|)|\nabla v-\nabla h|\,dx\\
&\leq\sigma\dashint_{\frac{1}{2}\Omega_{r}}\phi_{|\nabla v|}(|\nabla  v-\nabla h|)\omega^p_{B_r}\,dx\\
&\quad+c(\sigma)\dashint_{\frac{1}{2}\Omega_{r}}(\phi_{|\nabla v|})^{*}(\|\nabla\psi\|_{\infty}\phi'(|\nabla v|))\omega^{p}_{B_r}\,dx=:I_1+I_2
\end{split}
\end{align}
for any $\sigma\in(0,1)$. Then $I_1$ is absorbed to the left-hand side by choosing $\sigma$ sufficiently small depending on $n,p$ and $\Lambda$. To estimate $I_2$, we use \eqref{eq:equiv.phi}, $(\phi^{*})_{\phi'(|a|)}(\phi'(|a|))\eqsim\phi(|a|)$ and then \eqref{eq:T} to have
\begin{align}\label{eq:2comp7}
\begin{split}
I_2&\lesssim\dashint_{\frac{1}{2}\Omega_{r}}(\phi^*)_{\phi'(|\nabla v|)}(\|\nabla\psi\|_{\infty}\phi'(|\nabla v|))\omega^{p}_{B_r}\,dx\\
&\lesssim\|\nabla\psi\|_{\infty}^2\dashint_{\frac{1}{2}\Omega_{r}}(\phi^*)_{\phi'(|\nabla v|)}(\phi'(|\nabla v|))\omega^{p}_{B_r}\,dx\\
&\lesssim\|\nabla\psi\|_{\infty}^2\dashint_{\frac{1}{2}\Omega_{r}}\phi(|\nabla v|)\omega^p_{B_r}\,dx\\
&\lesssim\|\nabla\psi\|_{\infty}^2\dashint_{\frac{1}{2}\Omega_{r}}\phi(|T(x)\nabla v|)\omega^p_{B_r}\,dx.
\end{split}
\end{align}
Now, we use minimizing property of $\nabla v$ together with \eqref{eq:MB2}, \eqref{eq:T} and $|\Omega_r|\eqsim |\frac{1}{2}\Omega_r|$ to have
\begin{align}\label{eq:2comp9}
\begin{split}
\dashint_{\frac{1}{2}\Omega_{r}}\phi(|T(x)\nabla v|)\omega^p_{B_r}\,dx&\lesssim\dashint_{\frac{1}{2}\Omega_{r}}\phi(|T(x)\nabla h|)\omega^p_{B_r}\,dx\lesssim\dashint_{\Omega_{r}}\phi(|\nabla h|)\omega^p_{B_r}\,dx.
\end{split}
\end{align}
Summing up the above estimates, we obtain \eqref{eq:2comp}.
\end{proof}

\subsection{Decay Estimates at Boundary} Before providing decay estimates of $\mathcal{V}(\cdot,\nabla z)$, we discuss some regularity results and corresponding estimates related to $\tilde{v}$ and $v$ which are defined in \eqref{eq:2comp3} and \eqref{eq:2comp2}, respectively. First, we have the following estimates which imply Lipschitz regularity and $C^{1,\alpha}$ regularity of $\tilde{v}$.

\begin{proposition}\label{prop:regy}
	Assuming \eqref{eq:Omega}, let $\tilde{v}$ be the solution of \eqref{eq:2nd}. Then there holds
	\begin{align}\label{eq:lip.regy}
	\sup_{\frac{1}{32}B_r\cap\setR^n_+}|\nabla \tilde{v}|^{p}\omega^p_{B_r}\leq c\dashint_{\frac{1}{4}B_r\cap\setR^n_+}|\nabla \tilde{v}|^p\omega^p_{B_r}\,dy.	
	\end{align}
	Moreover, there exist $\alpha=\alpha(n,p,\Lambda)\in(0,1)$ and $c=c(n,p,\Lambda)>0$ such that
	\begin{align}\label{eq:decayy}
	\begin{split}
	&\dashint_{\lambda B_r\cap\setR^n_+}|\mathcal{V}_{B_r}(\nabla \tilde{v})-\mean{\mathcal{V}_{B_r}(\nabla \tilde{v})}_{\lambda {B_r}\cap\setR^n_+}|^2\,dy\\
	&\quad\quad\quad\quad\leq c\lambda^{2\alpha}\dashint_{\frac{1}{8}{B_r}\cap\setR^n_+}|\mathcal{V}_{B_r}(\nabla \tilde{v})-\mean{\mathcal{V}_{B_r}(\nabla \tilde{v})}_{\frac{1}{8}B_r\cap\setR^n_+}|^2\,dy
	\end{split}
	\end{align}
	holds for all $\lambda\in(0,\frac{1}{80})$.
\end{proposition}

\begin{proof}
We first show \eqref{eq:decayy}. Throughout the proof of \eqref{eq:decayy}, let us write the center of the ball $B_r$ as $y_{B_r}$. 
\\
\textbf{Step 1.}
If $\frac{1}{8}B_r\subset\{y\in\setR^n:y_n\geq 0\}$, then it directly follows from \cite[Proposition 15]{BDGN}. 
\\
\textbf{Step 2.}
Now, we consider $z_{B_r}\in\{y\in\setR^n:y_n=0\}$. We have by \eqref{eq:MB3} that
\begin{align*}
\Lambda^{-1}\omega_{B_r}\identity\leq\setM_{B_r}\leq\omega_{B_r}\identity.
\end{align*}
Since the equation \eqref{eq:2comp3} and estimate \eqref{eq:decayy} are invariant under normalization, without loss of generality we let $\omega_{B_r}=1$. Also, assume that ${B_r}$ is centered at 0, i.e., $z_{B_r}=0$.

Since $\setM_{B_r}$ is symmetric, there is an orthogonal matrix $Q$ and a diagonal matrix $\mathbb{D}_{B_r}$ such that $\setM_{B_r}=Q\mathbb{D}_{B_r}Q^*$. Then $\tilde{w}_0(y):=\tilde{v}(Qy)$ is a solution of \eqref{eq:2comp3} with $\mathbb{D}_{B_r}$ instead of $\setM_{B_r}$. Notice that the boundary of the domain is also rotated, and for $\mathbb{D}_{B_r}$ we have
\begin{align}\label{eq:D}
\Lambda^{-1}\identity\leq\mathbb{D}_{B_r}\leq \identity.
\end{align}
Now, we apply an anisotropic scaling $y\mapsto\mathbb{D}^{-1}_{{B_r}}y$. This turns estimates on half balls (with the rotated flat part) into estimates on half-ellipses (with  the rotated flat part) of uniformly bounded eccentricity depending on $n$ and $\Lambda$. Thus, after properly rotating the coordinate axis to make the rotated flat part to the subset of $\{y\in\setR^n:y_n=0\}$, we can take the odd extension to \eqref{eq:2comp3} to obtain \eqref{eq:decayy}.

In detail, let $\tilde{Q}$ be an $n\times n$ orthogonal matrix which maps $\mathbb{D}^{-1}_{{B_r}}Q^*(\{y\in\setR^n:y_n=0\})$ to $\{y\in\setR^n:y_n=0\}$. In other words, $\tilde{Q}$ satisfies 
\begin{align}\label{eq:mat2}
\tilde{Q}(\mathbb{D}^{-1}_{{B_r}}Q^*(\{y\in\setR^n:y_n=0\}))=\{y\in\setR^n:y_n=0\}.
\end{align}
Now, using this $n\times n$ orthogonal matrix $\tilde{Q}$, define
\begin{align}\label{eq:mat7}
\tilde{w}(y):=\tilde{v}(\tilde{Q}\mathbb{D}_{B_r}Q^*y).
\end{align}
Then we have
\begin{align}\label{eq:mat8}
(Q\mathbb{D}_{B_r}^{-1}(\tilde{Q})^*)^*\setM_{{B_r}}^2(Q\mathbb{D}_{B_r}^{-1}(\tilde{Q})^*)=\tilde{Q}\mathbb{D}^{-1}_{{B_r}}Q^*\setM^2_{{B_r}}Q\mathbb{D}^{-1}_{B_r}(\tilde{Q})^*=\tilde{Q}\mathbb{D}^{-1}_{{B_r}}\mathbb{D}_{{B_r}}^2\mathbb{D}^{-1}_{B_r}(\tilde{Q})^*=\identity
\end{align}
and for $t=Q\mathbb{D}^{-1}_{B_r}(\tilde{Q})^*y$, we have
\begin{align}\label{eq:mat9}
\tilde{w}(Q\mathbb{D}^{-1}_{B_r}(\tilde{Q})^*y)=\tilde{v}(y)=\tilde{w}(t)
\end{align}
and so
\begin{align}\label{eq:mat10}
\nabla_y \tilde{v}(y)=\nabla_y\tilde{w}(t)=(Q\mathbb{D}^{-1}_{{B_r}}(\tilde{Q})^*)\nabla_t\tilde{w}(t).
\end{align}
Therefore, \eqref{eq:2comp3} defined in $\Psi(\tfrac{1}{2}\Omega_r)\subset {B_r}$ transforms into
\begin{align*}
\begin{split}
&\divergence_y(|\setM_{B_r}\nabla \tilde{v}|^{p-2}\setM_{B_r}^2\nabla \tilde{v})=\divergence_y(\left<\setM_{B_r}\nabla \tilde{v},\setM_{B_r}\nabla \tilde{v}\right>^{\frac{p-2}{2}}\setM_{B_r}^2\nabla \tilde{v})\\
&=\divergence_t\left((Q\mathbb{D}^{-1}_{B_r}(\tilde{Q})^*)^*\left<\setM_{B_r}Q\mathbb{D}^{-1}_{B_r}(\tilde{Q})^*\nabla_t\tilde{w},\setM_{B_r}Q\mathbb{D}^{-1}_{B_r}(\tilde{Q})^*\nabla_t\tilde{w}\right>^{\frac{p-2}{2}}\setM_{B_r}^2Q\mathbb{D}^{-1}_{{B_r}}(\tilde{Q})^*\nabla_t\tilde{w}\right)\\
&=\divergence_t\left(\left<(\setM_{B_r}Q\mathbb{D}^{-1}_{B_r}(\tilde{Q})^*)^*\setM_{B_r}Q\mathbb{D}^{-1}_{B_r}(\tilde{Q})^*\nabla_t\tilde{w},\nabla_t\tilde{w}\right>^{\frac{p-2}{2}}\nabla_t\tilde{w}\right)\\
&=\divergence_t(|\nabla_t\tilde{w}|^{p-2}\nabla_t\tilde{w})=0
\end{split}
\end{align*}
defined in $t\in Q\mathbb{D}^{-1}_{B_r}(\tilde{Q})^*(\{y\in\setR^n:y_n=0\})$. Note that since $\Psi(\tfrac{1}{2}\Omega_r)\subset\{y\in\setR:y_n=0\}$ and \eqref{eq:mat2} hold, we can employ \cite{M1} and apply the odd extension for  $\divergence_t(|\nabla_t\tilde{w}|^{p-2}\nabla_t\tilde{w})=0$ and get the analogous estimate to \eqref{eq:decayy} for $w$ with half-ellipses instead of half-balls. Using the relation \eqref{eq:mat8} and $\setM_{B_r}=Q\mathbb{D}_{B_r}Q^*$ for changing $w$ to $v$, and then using the fact that all balls can be covered by slightly enlarged ellipses and vice versa, the estimate \eqref{eq:decayy} is also true for half balls. Then we have \eqref{eq:decayy}.

\textbf{Step 3.}
Now, we consider the general case, i.e., $\frac{1}{4}{B_r}\not\subset\{y\in\setR^n:y_n\geq 0\}$ and $z=z_{B_r}\not\in\{y\in\setR^n:y_n=0\}$ holds. We employ the argument of \cite[Lemma 3.7]{KZ1}. Denote $z=(z_1,\dots,z_{n-1},z_n)$, $\bar{z}=(z_1,\dots,z_{n-1},0)$, and recall that $0<\lambda< \frac{1}{80}$ and $z_n>0$ since $\tilde{x}\in\Omega$, where $\tilde{x}$ is the center of the ball $B_r$ with $4B_r\subset 2B_0$.
Let us specify the exact center of the balls in this step. In particular, we write ${B_r}={B_r}(z)$.

\textit{Case 1.} $z_n>\tfrac{r}{10}$. In this case we have
\begin{align}\label{eq:decay1}
\lambda {B_{r}}(z)\subset \tfrac{1}{10}B_{r}(z)\subset \tfrac{z_n}{r}B_{r}(z)\subset\setR^n_+.
\end{align}
By the interior estimates in \cite[Theorem 6.4]{DieSV09}, we obtain
\begin{align}\label{eq:decay2}
\begin{split}
\dashint_{{\lambda B_{r}}(z)}&|\mathcal{V}_{B_r}(\nabla \tilde{v})-\mean{\mathcal{V}_{B_r}(\nabla \tilde{v})}_{\lambda B_{r}(z)}|^2\,dy\\
&\lesssim\left(\dfrac{\lambda}{1/10}\right)^{2\alpha}\dashint_{\frac{1}{10}B_{r}(z)}|\mathcal{V}_{B_r}(\nabla \tilde{v})-\mean{\mathcal{V}_{B_r}(\nabla \tilde{v})}_{\frac{1}{10}B_{r}(z)}|^2\,dy\\
&\lesssim\left(\frac{\lambda}{1/10}\right)^{2\alpha}\left(\dfrac{1}{|\frac{1}{10}B_{r}(z)|}\int_{\frac{1}{4}B_{r}(z)\cap\setR^n_+}|\mathcal{V}_{B_r}(\nabla \tilde{v})-\mean{\mathcal{V}_{B_r}(\nabla \tilde{v})}_{\frac{1}{4}B_{r}(z)\cap\setR^n_+}|^2\,dy\right)\\
&\lesssim\lambda^{2\alpha}\dashint_{\frac{1}{4}B_{r}(z)\cap\setR^n_+}|\mathcal{V}_{B_r}(\nabla \tilde{v})-\mean{\mathcal{V}_{B_r}(\nabla \tilde{v})}_{\frac{1}{4}B_{r}(z)\cap\setR^n_+}|^2\,dy.
\end{split}
\end{align}
Thus we obtain \eqref{eq:decayy} in this case.

\textit{Case 2.} $0<z_n\leq \frac{r}{10}$. We divide the proof into two subcases.

\textit{Subcase 1.} $0<\lambda<\frac{z_n}{4r}$. In this subcase
\begin{align}\label{eq:decay3}
\lambda B_{r}(z)\subset \tfrac{z_n}{4r}B_{r}(z)\subset \tfrac{5z_n}{4r}B^+_{r}(\bar{z}).
\end{align}
By the interior estimates in \cite[Theorem 6.4]{DieSV09}, we have
\begin{align}\label{eq:decay4}
\begin{split}
\dashint_{\lambda B_{r}(z)}&|\mathcal{V}_{B_r}(\nabla \tilde{v})-\mean{\mathcal{V}_{B_r}(\nabla \tilde{v})}_{\lambda B_{r}(z)}|^2\,dy\\
&\lesssim\left(\dfrac{4\lambda}{z_n/4r}\right)^{2\alpha}\dashint_{\frac{z_n}{4r}B_{r}(z)}|\mathcal{V}_{B_r}(\nabla \tilde{v})-\mean{\mathcal{V}_{B_r}(\nabla \tilde{v})}_{\frac{z_n}{4r}B_{r}(z)}|^2\,dy\\
&\lesssim\left(\dfrac{4\lambda}{z_n/4r}\right)^{2\alpha}\dashint_{\frac{5z_n}{4r}B^+_{r}(\bar{z})}|\mathcal{V}_{B_r}(\nabla \tilde{v})-\mean{\mathcal{V}_{B_r}(\nabla \tilde{v})}_{\frac{5z_n}{4r}B^+_{r}(\bar{z})}|^2\,dy.
\end{split}
\end{align}
Since $\bar{z}\in\{y\in\setR^n:z_n=0\}$, by \textbf{Step 2} above and then using $0<z_n\leq\frac{r}{10}$, we have
\begin{align}\label{eq:decay5}
\begin{split}
\dashint_{\frac{5z_n}{4r}B^+_{r}(\bar{z})}&|\mathcal{V}_{B_r}(\nabla \tilde{v})-\mean{\mathcal{V}_{B_r}(\nabla \tilde{v})}_{\frac{5z_n}{4r}B^+_{r}(\bar{z})}|^2\,dy\\
&\lesssim \left(\dfrac{5z_n}{r}\right)^{2\alpha}\dashint_{\frac{1}{8}B^+_{r}(\bar{z})}|\mathcal{V}_{B_r}(\nabla \tilde{v})-\mean{\mathcal{V}_{B_r}(\nabla \tilde{v})}_{\frac{1}{8}B^+_{r}(\bar{z})}|^2\,dy\\
&\lesssim \left(\dfrac{5z_n}{r}\right)^{2\alpha}\dashint_{\frac{1}{4}B_{r}(z)\cap\setR^n_+}|\mathcal{V}_{B_r}(\nabla \tilde{v})-\mean{\mathcal{V}_{B_r}(\nabla \tilde{v})}_{\frac{1}{4}B_{r}(z)\cap\setR^n_+}|^2\,dy.
\end{split}
\end{align}
Combining the above two estimates, \eqref{eq:decay4} and \eqref{eq:decay5}, we obtain \eqref{eq:decayy} in this subcase.

\textit{Subcase 2.} $\lambda\geq \tfrac{z_n}{4r}$. Since $\lambda\leq \frac{1}{40}$, we see that
\begin{align}\label{eq:decay6}
\lambda B_{r}(z)\cap\setR^n_+\subset 5\lambda B^+_{r}(\bar{z})\subset \tfrac{1}{8}B^+_{r}(\bar{z})\subset \tfrac{1}{4}B_{r}(z)\cap\setR^n_+.
\end{align}
Therefore, using the boundary estimate above in \textbf{Step 2}, we have
\begin{align}\label{eq:decay7}
\begin{split}
\dashint_{\lambda B_{r}(z)\cap\setR^n_+}&|\mathcal{V}_{B_r}(\nabla \tilde{v})-\mean{\mathcal{V}_{B_r}(\nabla \tilde{v})}_{\lambda B_{r}(z)\cap\setR^n_+}|^2\,dy\\
&\lesssim\dashint_{20\lambda B^+_{r}(\bar{z})}|\mathcal{V}_{B_r}(\nabla \tilde{v})-\mean{\mathcal{V}_{B_r}(\nabla \tilde{v})}_{20\lambda B^+_{r}(\bar{z})}|^2\,dy\\
&\lesssim\left(\dfrac{5\lambda}{1/8}\right)^{2\alpha}\dashint_{\frac{1}{8}B^+_{r}(\bar{z})}|\mathcal{V}_{B_r}(\nabla \tilde{v})-\mean{\mathcal{V}_{B_r}(\nabla \tilde{v})}_{\frac{1}{8}B^+_{r}(\bar{z})}|^2\,dy\\
&\lesssim\lambda^{2\alpha}\dashint_{\frac{1}{4}B_{r}(z)\cap\setR^n_+}|\mathcal{V}_{B_r}(\nabla \tilde{v})-\mean{\mathcal{V}_{B_r}(\nabla \tilde{v})}_{\frac{1}{4}B_{r}(z)\cap\setR^n_+}|^2\,dy.
\end{split}
\end{align}
Merging all cases \textit{Case 1}--\textit{Case 2}, we have \eqref{eq:decayy} in \textbf{Step 3}. Therefore, by \textbf{Step 1}--\textbf{Step 3}, we have \eqref{eq:decayy}.

To show \eqref{eq:lip.regy}, we employ the similar argument as above. If $\frac{1}{4}{B_r}\subset\{y\in\setR^n:y_n\geq 0\}$, then it follows from \cite[Proposition 15]{BDGN}. Now, when  $z_{{B_r}}\in\{y\in\setR^n:y_n=0\}$, by employing the same matrix $\tilde{Q}$, $\mathbb{D}_{B_r}$, $Q^*$ as above, we can apply \cite[Lemma 5]{L1} and we have \eqref{eq:lip.regy} in this case. In the general case, i.e., when $\frac{1}{4}{B_r}\not\subset\{y\in\setR^n:y_n\geq 0\}$ and $z_{B_r}\not\in\{y\in\setR^n:y_n=0\}$ holds, we divide the cases as same as above and apply the argument of \cite{L2} instead of \cite{DieSV09}. Now, \eqref{eq:lip.regy} is obtained.
\end{proof}

Now, we transform the above estimates for $\tilde{v}$ to the estimates for $v$.

\begin{proposition}\label{prop:regx}
	Assuming \eqref{eq:Omega}, let $v$ be the solution of \eqref{eq:2comp2}. Then there holds
	\begin{align}\label{eq:lip.regx}
	\sup_{\frac{1}{32}\Omega_r}|\nabla v|^{p}\omega^p_{{B_r}}\leq c\dashint_{\frac{1}{4}\Omega_r}|\nabla v|^p\omega^p_{{B_r}}\,dx.	
	\end{align}
	Moreover, there exist $\alpha=\alpha(n,p,\Lambda)\in(0,1)$ and $c=c(n,p,\Lambda)>0$ such that
	\begin{align}\label{eq:decayx}
	\begin{split}
	\dashint_{\lambda \Omega_r}&|\mathcal{V}_{B_r}(\nabla v)-\mean{\mathcal{V}_{B_r}(\nabla v)}_{\lambda \Omega_r}|^2\,dx\\
	&\leq c\lambda^{2\alpha}\dashint_{\frac{1}{4}\Omega_r}|\mathcal{V}_{{B_r}}(\nabla v)-\mean{\mathcal{V}_{{B_r}}(\nabla v)}_{\frac{1}{4}\Omega_r}|^2\,dx+c\|\nabla\psi\|^2_{\infty}\lambda^{-n}\dashint_{\frac{1}{4}\Omega_r}|\nabla v|^p\omega^p_{{B_r}}\,dx
	\end{split}
	\end{align}
	holds for all $\lambda\in(0,\frac{1}{80})$.
\end{proposition}

\begin{proof}
To obtain \eqref{eq:decayx}, using \eqref{eq:Psi.ball} with the help of \eqref{eq:Omega}, there holds
\begin{align}\label{eq:decay9}
\begin{split}
\dashint_{\lambda\Omega_r}&|\mathcal{V}_{B_r}(\nabla v)-\mean{\mathcal{V}_{B_r}(\nabla v)}_{\lambda \Omega_r}|^2\,dx\\
&=\dashint_{\Psi(\lambda \Omega_r)}|\mathcal{V}_{{B_r}}((\nabla\Psi)\nabla\tilde{v})-\mean{\mathcal{V}_{{B_r}}((\nabla\Psi)\nabla\tilde{v})}_{\Psi(\lambda \Omega_r)}|^2\,dy\\
&\lesssim\dashint_{\Psi(4\lambda \Omega_r)}|\mathcal{V}_{{B_r}}((\nabla\Psi)\nabla\tilde{v})-\mean{\mathcal{V}_{{B_r}}(\nabla\tilde{v})}_{\Psi(4\lambda \Omega_r)}|^2\,dy\\
&\lesssim\dashint_{\Psi(4\lambda \Omega_r)}|\mathcal{V}_{{B_r}}((\nabla\Psi)\nabla\tilde{v})-\mathcal{V}_{{B_r}}(\nabla\tilde{v})|^2\,dy\\
&\quad+\dashint_{\Psi(4\lambda \Omega_r)}|\mathcal{V}_{{B_r}}(\nabla\tilde{v})-\mean{\mathcal{V}_{{B_r}}(\nabla\tilde{v})}_{\Psi(4\lambda \Omega_r)}|^2\,dy=:I_1+I_2.
\end{split}
\end{align}
For $I_1$, by Lemma \ref{lem:equivA} and Lemma \ref{lem:diff.A2}, we obtain
\begin{align}\label{eq:decay10}
\begin{split}
I_1&\lesssim\dashint_{\Psi(4\lambda \Omega_r)}|\mathcal{A}((\nabla\Psi)\nabla\tilde{v})-\mathcal{A}_{B_r}(\nabla\tilde{v})|\cdot|(\nabla\Psi)\nabla\tilde{v}-\nabla\tilde{v}|\,dy\\
&\lesssim \dashint_{\Psi(4\lambda \Omega_r)}\omega^p_{{B_r}}\|\nabla\psi\|_{\infty}|(\nabla\Psi)\nabla\tilde{v}|^{p-1}\cdot\|\nabla\psi\|_{\infty}|\nabla\tilde{v}|\,dy\\
&\lesssim \|\nabla\psi\|_{\infty}^2\dashint_{\Psi(4\lambda \Omega_r)}\omega^p_{{B_r}}|(\nabla\Psi)\nabla\tilde{v}|^{p}|(\nabla\Psi)^{-1}|\,dy\\
&\lesssim \|\nabla\psi\|_{\infty}^2\dashint_{4\lambda \Omega_r}|\nabla v|^{p}\omega^p_{{B_r}}\,dx\\
&\lesssim \|\nabla\psi\|_{\infty}^2\lambda^{-n}\dashint_{\frac{1}{4} \Omega_r}|\nabla v|^{p}\omega^p_{{B_r}}\,dx.
\end{split}
\end{align}
On the other hand, for $I_2$, we apply \eqref{eq:Psi.ball} from \eqref{eq:Omega}, and \eqref{eq:decayy} to have
\begin{align}\label{eq:decay11}
\begin{split}
I_2&\lesssim \dashint_{8\lambda {B_r}\cap\setR^n_+}|\mathcal{V}_{{B_r}}(\nabla \tilde{v})-\mean{\mathcal{V}_{B_r}(\nabla\tilde{v})}_{8\lambda {B_r}\cap\setR^n_+}|^2\,dy\\
&\lesssim \lambda^{2\alpha}\dashint_{\frac{1}{8}{B_r}\cap\setR^{n}_+}|\mathcal{V}_{{B_r}}(\nabla \tilde{v})-\mean{\mathcal{V}_{{B_r}}(\nabla \tilde{v})}_{\frac{1}{8}{B_r}\cap\setR^{n}_+}|^2\,dy\\
&\lesssim \lambda^{2\alpha}\dashint_{\Psi^{-1}(\frac{1}{8}{B_r}\cap\setR^n_+)}|\mathcal{V}_{{B_r}}((\nabla\Psi)^{-1}\nabla v)-\mean{\mathcal{V}_{{B_r}}(\nabla v)}_{\frac{1}{4}\Omega_r}|^2\,dx\\
&\lesssim \lambda^{2\alpha}\dashint_{\frac{1}{4}\Omega_r}|\mathcal{V}_{{B_r}}((\nabla\Psi)^{-1}\nabla v)-\mathcal{V}_{{B_r}}(\nabla v)|^2\,dx\\
&\quad+ \lambda^{2\alpha}\dashint_{\frac{1}{4}\Omega_r}|\mathcal{V}_{{B_r}}(\nabla v)-\mean{\mathcal{V}_{{B_r}}(\nabla v)}_{\frac{1}{4}\Omega_r}|^2\,dx\\
&=:\lambda^{2\alpha}(I_{2,1}+I_{2,2}).
\end{split}
\end{align}
To obtain \eqref{eq:decayx}, we only have to estimate $I_{2,1}$. By the similar argument as \eqref{eq:decay10}, we have
\begin{align}\label{eq:decay12}
\begin{split}
I_{2,1}&\lesssim\dashint_{\frac{1}{4}\Omega_r}|\mathcal{A}((\nabla\Psi)^{-1}\nabla v)-\mathcal{A}_{B_r}(\nabla v)|\cdot\|\nabla\psi\|_{\infty}|\nabla v|\,dx\\
&\lesssim \|\nabla\psi\|_{\infty}^2\dashint_{\frac{1}{4}\Omega_r}\omega^p_{{B_r}}|(\nabla\Psi)^{-1}\nabla v|^{p}|\nabla\Psi|\,dx\\
&\lesssim \|\nabla\psi\|_{\infty}^2\dashint_{\frac{1}{4}\Omega_r}|\nabla v|^{p}\omega^p_{{B_r}}\,dx.
\end{split}
\end{align}
Summing up all estimates \eqref{eq:decay9}--\eqref{eq:decay12}, we conclude \eqref{eq:decayx}.

To show \eqref{eq:lip.regx}, by using \eqref{eq:T} and \eqref{eq:Psi.ball}, there holds
\begin{align}\label{eq:decay13}
\begin{split}
\sup_{x\in\frac{1}{32}\Omega_r}|\nabla v(x)|^{p}\omega^p_{{B_r}}\leq \sup_{y\in\Psi(\frac{1}{32}\Omega_r)}|(\nabla\Psi)\nabla \tilde{v}(y)|^{p}\omega^p_{{B_r}}\lesssim \sup_{y\in\frac{1}{16}{B_r}\cap\setR^n_+}|\nabla \tilde{v}(y)|^{p}\omega^p_{{B_r}}
\end{split}
\end{align}
and by \eqref{eq:lip.regy}, we have
\begin{align}\label{eq:decay14}
\begin{split}
\sup_{y\in\frac{1}{16}{B_r}\cap\setR^n_+}|\nabla \tilde{v}(y)|^{p}\omega^p_{{B_r}}&\lesssim\dashint_{\frac{1}{8}{B_r}\cap\setR^n_+}|\nabla \tilde{v}|^p\omega^p_{{B_r}}\,dy\\
&\lesssim\dashint_{\Psi^{-1}(\frac{1}{8}{B_r}\cap\setR^n_+)}|(\nabla\Psi)^{-1}\nabla v|^p\omega^p_{{B_r}}\,dx\\
&\lesssim\dashint_{\frac{1}{4}\Omega_r}|\nabla v|^p\omega^p_{{B_r}}\,dx.
\end{split}
\end{align}
Summing up the above two inequalities, we have \eqref{eq:lip.regx}.
\end{proof}

With the help of the above estimates for $v$, we give the following decay estimate of $\mathcal{V}(\cdot,\nabla z)$. Recall that for $B_0=B_R(x_0)$ with $x_0\in\partial\Omega$, let $B_r=B_r(\tilde{x})$ with $\tilde{x}\in\overline{\Omega}$ and $4{B_r}\subset 2B_0$. Also, $z$ on $B_0\cap\Omega$ is such that $z:=u\eta^{p'}$ and we take the zero extension for $z$ on $\setR^n\setminus (B_0\cap\Omega)$, if necessary.

\begin{proposition}[Decay estimate at boundary]\label{prop:dec} Let $z$, $u$ and $F$ be as in \eqref{eq:plap4} and \eqref{eq:z}. There exist $\lambda=\lambda(n,p,\Lambda)\in(0,\frac{1}{80})$, $s=s(n,p,\Lambda)>1$ and $\kappa_5=\kappa_5(n,p,\Lambda)\in(0,1)$ such that the following holds: If $\lvert\log\setM|_{\setBMO(4{B_r})}\leq\kappa_5$, and \eqref{eq:Omega} hold, then for every $\epsilon\in(0,1)$ there holds
\begin{align}\label{eq:dec}
\begin{split}
\dashint_{\lambda {B_r}}&|\mathcal{V}(x,\nabla z)-\mean{\mathcal{V}(\cdot,\nabla z)}_{\lambda {B_r}}|^2\,dx\\
&\qquad\leq\dfrac{1}{4}\dashint_{{B_r}}|\mathcal{V}(x,\nabla z)-\mean{\mathcal{V}(\cdot,\nabla z)}_{{B_r}}|^2\,dx\\ &\qquad\quad+c\left(\lvert\log\setM|^2_{\setBMO({B_r})}+\|\nabla\psi\|^2_{\infty}+\epsilon\right)\left(\dashint_{{B_r}}|\mathcal{V}(x,\nabla z)|^{2s}\,dx\right)^\frac{1}{s}\\
&\qquad\quad+c\,C^*(\epsilon)\left(\dashint_{2{B_r}}\left(\dfrac{\chi_{2B_0\cap\Omega}|u|^p}{R^p}\omega^p\right)^s\,dx\right)^{\frac{1}{s}}\\
&\qquad\quad+c\,C^*(\epsilon)\left(\dashint_{2{B_r}}|\chi_{2B_0\cap\Omega}\mathcal{V}(x,F)|^{2s}\,dx\right)^{\frac{1}{s}}
\end{split}
\end{align}
with $c=c(n,p,\Lambda)$ and $C^*(\epsilon)$ defined in \eqref{eq:C}.
\end{proposition}
\begin{proof}
We first assume $\lvert\log\setM|_{\setBMO(4{B_r})}\leq\kappa_4$ with $\kappa _4$ from Proposition \ref{prop:comp2}. Since $z=|\nabla z|=0$ in $\lambda {B_r}\setminus\Omega$ and $|\lambda\Omega_r|\eqsim |\lambda {B_r}|$ hold from \eqref{eq:m.den1} because of \eqref{eq:Omega} and $\tilde{x}\in\overline{\Omega}$, we have
\begin{align}\label{eq:dec1}
\begin{split}
I_1&:=\dashint_{\lambda {B_r}}|\mathcal{V}(x,\nabla z)-\mean{\mathcal{V}(\cdot,\nabla z)}_{\lambda {B_r}}|^2\,dx\\
&\lesssim\dashint_{\lambda {B_r}}|\mathcal{V}(x,\nabla z)-\mean{\mathcal{V}(\cdot,\nabla v)}_{\lambda \Omega_r}|^2\,dx\\
&\lesssim\dashint_{\lambda \Omega_r}|\mathcal{V}(x,\nabla z)-\mean{\mathcal{V}(\cdot,\nabla v)}_{\lambda \Omega_r}|^2\,dx\\
&\lesssim\dashint_{\lambda \Omega_r}|\mathcal{V}(x,\nabla v)-\mean{\mathcal{V}(\cdot,\nabla v)}_{\lambda \Omega_r}|^2\,dx+\dashint_{\lambda \Omega_r}|\mathcal{V}(x,\nabla z)-\mathcal{V}(x,\nabla v)|^2\,dx\\
&=:I_2+I_3.
\end{split}
\end{align}

We start to estimate $I_3$. There holds
\begin{align}\label{eq:dec2}
\begin{split}
I_3&\lesssim\lambda^{-n}\dashint_{\frac{1}{2}\Omega_r}|\mathcal{V}(x,\nabla z)-\mathcal{V}(x,\nabla v)|^2\,dx\\
&\lesssim\lambda^{-n}\dashint_{\frac{1}{2}\Omega_r}|\mathcal{V}_{B_r}(\nabla z)-\mathcal{V}_{B_r}(\nabla v)|^2\,dx\\
&\quad +\lambda^{-n}\dashint_{\frac{1}{2}\Omega_r}|\mathcal{V}(x,\nabla z)-\mathcal{V}_{B_r}(\nabla z)|^2\,dx+\lambda^{-n}\dashint_{\frac{1}{32}\Omega_r}|\mathcal{V}(x,\nabla v)-\mathcal{V}_{B_r}(\nabla v)|^2\,dx\\
&=:I_{3,1}+I_{3,2}+I_{3,3}.
\end{split}
\end{align}

For $I_{3,2}$ and $I_{3,3}$, since $\tilde{x}\in\overline{\Omega}$ and \eqref{eq:Omega} holds, $|\Omega_r|\eqsim|{B_r}|$ also holds. Then we can apply the similar argument of the proof as in Proposition 18 in \cite{BDGN}. By \eqref{eq:M2}, \eqref{eq:MB2}, \eqref{eq:MB3}, Lemma \ref{lem:equivA}, \eqref{eq:equiv.phi}, and H\"{o}lder's inequality with exponents $(2s',s,2s')$, we have
\begin{align}\label{eq:dec4}
\begin{split}
I_{3,2}&\lesssim \lambda^{-n}\left(\dashint_{\Omega_r}\left(\dfrac{|\setM_{{B_r}}-\setM|}{|\setM_{{B_r}}|}\right)^{4s'}\,dx\right)^{\frac{1}{2s'}}\left(\dashint_{\Omega_r}(|\nabla z|^p\omega^p)^s\,dx\right)^{\frac{1}{s}}\\
&\quad\quad\times\left(1+\left(\dashint_{\Omega_r}\left(\dfrac{\omega^p_{{B_r}}}{\omega^p}\right)^{2s'}\,dx\right)^{\frac{1}{2s'}}\right).
\end{split}
\end{align}
Together with $z=|\nabla z|=0$ in ${B_r}\setminus\Omega$ and $|\Omega_r|\eqsim|{B_r}|$ due to \eqref{eq:Omega}, we obtain
\begin{align}\label{eq:dec6}
\begin{split}
I_{3,2}\lesssim \lambda^{-n}\left(\dashint_{\Omega_r}(|\nabla z|^p\omega^p)^s\,dx\right)^{\frac{1}{s}}.
\end{split}
\end{align}
Note that by $4{B_r}\subset2B_0$, $\chi_{2B_0\cap\Omega}$ can be inserted in the integrand of the above estimate. On the other hand, similar to $I_{3,2}$, we continue to estimate $I_{3,3}$ as follows:
\begin{align}\label{eq:dec7}
\begin{split}
I_{3,3}&=\lambda^{-n}\dashint_{\frac{1}{32}\Omega_r}|V(\setM\nabla v)-V(\setM_{{B_r}}\nabla v)|^2\,dx\\
&\lesssim\lambda^{-n}\dashint_{\frac{1}{32}\Omega_r}\left(\dfrac{|\setM-\setM_{B_r}|}{|\setM_{{B_r}}|}\right)^2\left(|\nabla v|^p\omega^p+|\nabla v|^p\omega^p_{{B_r}}\right)\,dx\\
&\lesssim\lambda^{-n}\left(\sup_{\frac{1}{32}\Omega_r}|\nabla v|^p\omega^p_{{B_r}}\right)\left(\dashint_{\frac{1}{2}\Omega_r}\left(\dfrac{|\setM-\setM_{{B_r}}|}{|\setM_{{B_r}}|}\right)^2\left(\dfrac{\omega^p}{\omega^p_{{B_r}}}+1\right)\,dx\right)=\lambda^{-n}I_{3,3,1}I_{3,3,2}.
\end{split}
\end{align}
Here, \eqref{eq:lip.regx} in Proposition \ref{prop:regx}, \eqref{eq:T} and \eqref{eq:MB2} imply
\begin{align}\label{eq:dec8}
\begin{split}
I_{3,3,1}&:=\sup_{\frac{1}{32}\Omega_r}|\nabla v|^p\omega^p_{{B_r}}\\
&\lesssim\dashint_{\frac{1}{4}\Omega_r}|\nabla v|^p\omega^p_{{B_r}}\,dx\lesssim\dashint_{\frac{1}{4}\Omega_r}|\setM_{B_r}T(x)\nabla v|^p|T^{-1}(x)|^p\,dx\lesssim \dashint_{\frac{1}{2}\Omega_r}|\setM_{B_r}T(x)\nabla v|^p\,dx
\end{split}
\end{align}
and the minimizing property of $v$ and $h$, together with \eqref{eq:T} and \eqref{eq:MB2} give us that
\begin{align}\label{eq:dec9}
I_{3,3,1}\lesssim \dashint_{\frac{1}{2}\Omega_r}|\setM_{B_r}T(x)\nabla h|^p\,dx\lesssim\dashint_{\Omega_r}|\setM_{B_r}\nabla h|^p\,dx\lesssim \dashint_{\Omega_r}|\nabla z|^p\omega_{B_r}^p\,dx.
\end{align}
Then we use H\"{o}lder's inequality, $|\Omega_r|\eqsim|{B_r}|$ from \eqref{eq:Omega}, Lemma \ref{lem:Mu.weight} and $z=|\nabla z|=0$ in ${B_r}\setminus\Omega$ to obtain
\begin{align}\label{eq:dec10}
\begin{split}
I_{3,3,1}\lesssim\left(\dashint_{\Omega_{r}}(|\nabla z|^p\omega^p)^s\,dx\right)^{\frac{1}{s}}\left(\dashint_{\Omega_{r}}\left(\dfrac{\omega^p_{B_r}}{\omega^p}\right)^{\frac{s}{s-1}}\,dx\right)^{1-\frac{1}{s}}\lesssim\left(\dashint_{{B_r}}(|\nabla z|^p\omega^p)^s\,dx\right)^{\frac{1}{s}},
\end{split}
\end{align}
provided $\lvert\log\setM|_{\setBMO(4{B_r})}\leq\kappa_4$ holds. 
On the other hand, by H\"{o}lder's inequality, Lemma \ref{lem:BMOM} and Lemma \ref{lem:Mu.weight} we obtain
\begin{align}\label{eq:dec12}
\begin{split}
I_{3,3,2}&\lesssim\left(\dashint_{\frac{1}{2}\Omega_r}\left(\dfrac{|\setM-\setM_{{B_r}}|}{|\setM_{{B_r}}|}\right)^4\,dx\right)^{\frac{1}{2}}\left(\left(\dashint_{\frac{1}{2}\Omega_r}\left(\dfrac{\omega^p}{\omega^p_{{B_r}}}\right)^2\,dx\right)^{\frac{1}{2}}+1\right)\lesssim\lvert\log\setM|^2_{\setBMO({B_r})}.
\end{split}
\end{align}
Summing up, there holds
\begin{align}\label{eq:dec13}
\begin{split}
I_{3,3}&\lesssim\lambda^{-n}\lvert\log\setM|^2_{\setBMO({B_r})}\left(\dashint_{{B_r}}|\mathcal{V}(x,\nabla z)|^{2s}\,dx\right)^{\frac{1}{s}}.
\end{split}
\end{align}
For $I_{3,1}$, we first apply Proposition \ref{prop:comp} and Proposition \ref{prop:comp2}, use $z=|\nabla z|=0$ in $2{B_r}\setminus\Omega$ and $|2\Omega_r|\eqsim|2{B_r}|$ with the help of \eqref{eq:Omega}, and argue similarly to \eqref{eq:dec9}--\eqref{eq:dec10} for the integral of $|\nabla h|^p\omega^p_{{B_r}}$ term. The resulting estimate is as follows: 
\begin{align*}
\begin{split}
I_{3,1}&\lesssim \lambda^{-n}\dashint_{\frac{1}{2}\Omega_r}|\mathcal{V}_{B_r}(\nabla z)-\mathcal{V}_{B_r}(\nabla h)|^2\,dx+\lambda^{-n}\dashint_{\frac{1}{2}\Omega_r}|\mathcal{V}_{B_r}(\nabla h)-\mathcal{V}_{B_r}(\nabla v)|^2\,dx\\
&\lesssim \lambda^{-n}(\lvert\log\setM|^2_{\setBMO({B_r})}+\epsilon)\dashint_{\Omega_r}(|\nabla z|^p\omega^p)\,dx+\lambda^{-n}\|\nabla\psi\|^2_{\infty}\dashint_{\Omega_r}(|\nabla h|^p\omega_{B_r}^p)\,dx\\
&\quad+\lambda^{-n}C^*(\epsilon)\left(\dashint_{2\Omega_{r}}\left(\dfrac{|u|^p}{R^p}\omega^p\right)^s\,dx\right)^{\frac{1}{s}}+\lambda^{-n}C^*(\epsilon)\left(\dashint_{2\Omega_{r}}|\mathcal{V}(x,F)|^{2s}\,dx\right)^{\frac{1}{s}}\\
&\lesssim \lambda^{-n}(\lvert\log\setM|^2_{\setBMO({B_r})}+\|\nabla\psi\|^2_{\infty}+\epsilon)\left(\dashint_{{B_r}}|\mathcal{V}(x,\nabla z)|^{2s}\,dx\right)^{\frac{1}{s}}\\
&\quad+\lambda^{-n}C^*(\epsilon)\left(\dashint_{2{B_r}}\left(\dfrac{\chi_{2B_0\cap\Omega}|u|^p}{R^p}\omega^p\right)^s\,dx\right)^{\frac{1}{s}}+\lambda^{-n}C^*(\epsilon)\left(\dashint_{2{B_r}}|\chi_{2B_0\cap\Omega}\mathcal{V}(x,F)|^{2s}\,dx\right)^{\frac{1}{s}}.
\end{split}
\end{align*}
Consequently, we have
\begin{align}\label{eq:dec15}
\begin{split}
I_{3}&\lesssim\lambda^{-n}\left(\lvert\log\setM|^2_{\setBMO(4{B_r})}+\|\nabla\psi\|^2_{\infty}+\epsilon\right)\left(\dashint_{{B_r}}|\mathcal{V}(x,\nabla z)|^{2s}\,dx\right)^{\frac{1}{s}}\\
&\quad+\lambda^{-n}C^*(\epsilon)\left(\dashint_{2{B_r}}\left(\dfrac{\chi_{2B_0\cap\Omega}|u|^p}{R^p}\omega^p\right)^s\,dx\right)^{\frac{1}{s}}\\
&\quad+\lambda^{-n}C^*(\epsilon)\left(\dashint_{2{B_r}}|\chi_{2B_0\cap\Omega}\mathcal{V}(x,F)|^{2s}\,dx\right)^{\frac{1}{s}}.
\end{split}
\end{align}

For $I_2$, we have
\begin{align}\label{eq:dec16}
\begin{split}
I_2&\lesssim\dashint_{\lambda \Omega_r}|\mathcal{V}_{{B_r}}(\nabla v)-\mean{\mathcal{V}_{{B_r}}(\nabla v)}_{\lambda \Omega_r}|^2\,dx+\dashint_{\lambda \Omega_r}|\mathcal{V}(x,\nabla v)-\mathcal{V}_{{B_r}}(\nabla v)|^2\,dx\\
&\lesssim\dashint_{\lambda \Omega_r}|\mathcal{V}_{{B_r}}(\nabla v)-\mean{\mathcal{V}_{{B_r}}(\nabla v)}_{\lambda \Omega_r}|^2\,dx+\lambda^{-n}\dashint_{\frac{1}{2}\Omega_r}|\mathcal{V}(x,\nabla v)-\mathcal{V}_{{B_r}}(\nabla v)|^2\,dx\\
&=:I_{2,1}+I_{2,2}.
\end{split}
\end{align}
With the help of \eqref{eq:decayx} in Proposition \ref{prop:regx}, it follows that
\begin{align}\label{eq:dec17}
I_{2,1}\lesssim\lambda^{2\alpha}\dashint_{\frac{1}{4}\Omega_r}|\mathcal{V}_{{B_r}}(\nabla v)-\mean{\mathcal{V}_{{B_r}}(\nabla v)}_{\frac{1}{4}\Omega_r}|^2\,dx+\|\nabla\psi\|^2_{\infty}\lambda^{-n}\dashint_{\frac{1}{4}\Omega_r}|\nabla v|^p\omega^p_{{B_r}}\,dx.
\end{align}
Triangle inequalities yield
\begin{align}\label{eq:dec18}
\begin{split}
I_{2,1}&\lesssim\lambda^{2\alpha}\dashint_{\Omega_r}|\mathcal{V}(x,\nabla z)-\mean{\mathcal{V}(\cdot,\nabla z)}_{{B_r}}|^2\,dx+\lambda^{2\alpha}\dashint_{\frac{1}{2}\Omega_r}|\mathcal{V}_{{B_r}}(\nabla z)-\mathcal{V}_{{B_r}}(\nabla v)|^2\,dx\\
&\quad+\lambda^{2\alpha}\dashint_{\frac{1}{2}\Omega_r}|\mathcal{V}(x,\nabla z)-\mathcal{V}_{{B_r}}(\nabla z)|^2\,dx+\lambda^{2\alpha}\dashint_{\frac{1}{2}\Omega_r}|\mathcal{V}(x,\nabla v)-\mathcal{V}_{{B_r}}(\nabla v)|^2\,dx\\
&\quad+\, \|\nabla\psi\|^2_{\infty}\lambda^{-n}\dashint_{\frac{1}{4}\Omega_r}|\nabla v|^p\omega^p_{{B_r}}\,dx\\
&=:I_{2,1,0}+I_{2,1,1}+I_{2,1,2}+I_{2,1,3}+I_{2,1,4}.
\end{split}
\end{align}
To estimate $I_{2,1,0}$, by $z=|\nabla z|=0$ in ${B_r}\setminus\Omega$ and $|\Omega_r|\eqsim|{B_r}|$ due to \eqref{eq:Omega}, we have
\begin{align}\label{eq:dec19}
I_{2,1,0}\lesssim\lambda^{2\alpha}\dashint_{{B_r}}|\mathcal{V}(x,\nabla z)-\mean{\mathcal{V}(x,\nabla z)}_{{B_r}}|^2\,dx.
\end{align}
Besides, using the similar argument for $I_{3,1},I_{3,2}$ and $I_{3,3}$, we estimate $I_{2,1,1}, I_{2,1,2}$ and $I_{2,1,3}$, respectively, with replacing the factor $\lambda^{-n}$ with $\lambda^{2\alpha}$. For $I_{2,1,4}$, by the same argument as in \eqref{eq:dec8}--\eqref{eq:dec10}, we have 
\begin{align}\label{eq:dec20}
\begin{split}
I_{2,1,4}&\lesssim \|\nabla\psi\|^2_{\infty}\lambda^{-n}\left( \dashint_{{B_r}}|\mathcal{V}(x,\nabla z)|^{2s}\,dx\right)^{\frac{1}{s}}.
\end{split}
\end{align}
On the other hand, for $I_{2,2}$ we apply the same estimate for $I_{3,3}$. Finally, we have
\begin{align}\label{eq:dec21}
\begin{split}
I_1&\leq I_2+I_3\\
&\leq c\lambda^{2\alpha}\dashint_{{B_r}}|\mathcal{V}(x,\nabla z)-\mean{\mathcal{V}(\cdot,\nabla z)}_{{B_r}}|^2\,dx\\
&\quad+c\lambda^{-n}\left(\lvert\log\setM|^2_{\setBMO(4{B_r})}+\|\nabla\psi\|^2_{\infty}+\epsilon\right)\left(\dashint_{{B_r}}|\mathcal{V}(x,\nabla z)|^{2s}\,dx\right)^{\frac{1}{s}}\\
&\quad+c\lambda^{-n}C^*(\epsilon)\left(\dashint_{2{B_r}}\left(\dfrac{\chi_{2B_0\cap\Omega}|u|^p}{R^p}\omega^p\right)^s\,dx\right)^{\frac{1}{s}}\\
&\quad+c\lambda^{-n}C^*(\epsilon)\left(\dashint_{2{B_r}}|\chi_{2B_0\cap\Omega}\mathcal{V}(x,F)|^{2s}\,dx\right)^{\frac{1}{s}}
\end{split}
\end{align}
for some $c=c(n,p,\Lambda)$. We select a small $\lambda=\lambda(n,p,\Lambda)\in(0,\frac{1}{80})$ such that $c\lambda^{2\alpha}\leq\frac{1}{4}$ holds, so that we get \eqref{eq:dec}.
\end{proof}

We define the Hardy-Littlewood maximal function and the sharp maximal function for $f\in L^{1}_{\loc}$ and $\rho\in[1,\infty)$ by
\begin{align}\label{eq:def.max}
\mathcal{M}_{\rho}f(x):=\sup_{r>0}\left(\dashint_{B_r(x)}|f|^{\rho}\,dy\right)^{\frac{1}{\rho}},\quad\mathcal{M}^{\sharp}_{\rho}f(x):=\sup_{r>0}\left(\dashint_{B_r(x)}|f-\mean{f}_{B_r(x)}|^{\rho}\,dy\right)^{\frac{1}{\rho}}.
\end{align}
Now, we employ Proposition \ref{prop:dec} to show the pointwise sharp maximal function estimate, which is more adaptable form to our gradient estimates. Recall that for $B_0=B_R(x_0)$ with $x_0\in\partial\Omega$.

\begin{proposition}\label{prop:sharpmax}
Let $z$, $u$ and $F$ be as in \eqref{eq:plap4} and \eqref{eq:z}. There exists $s=s(n,p,\Lambda)>1$ and $\kappa_5=\kappa_5(n,p,\Lambda)$ such that the following holds: If $\lvert\log\setM|_{\setBMO(4B_0)}\leq\kappa_5$ and \eqref{eq:Omega} hold, then for a.e. $x\in\setR^n$ and any $\epsilon\in(0,1]$, there holds
\begin{align}\label{eq:sharpmax}
\begin{split}
\mathcal{M}^{\sharp}_2\left(\mathcal{V}(\cdot,\nabla z)\right)(x)&\leq c\left(\lvert\log\setM|_{\setBMO(4B_0)}+\|\nabla\psi\|_{\infty}+\epsilon\right)\mathcal{M}_{2s}\left(\mathcal{V}(\cdot,\nabla z)\right)(x)\\
&\quad+c\,C^*(\epsilon^2)R^{-\frac{p}{2}}\left(\mathcal{M}_{s}\left(\chi_{4B_0\cap\Omega}|u|^p\omega^p\right)(x)\right)^{\frac{1}{2}}\\
&\quad+c\,C^*(\epsilon^2)\mathcal{M}_{2s}\left(\chi_{4B_0\cap\Omega}\mathcal{V}(\cdot,F)\right)(x)\\
&\quad+c\,\dfrac{R^n}{(R+|x-x_0|)^n}\left(\dashint_{B_0}|\mathcal{V}(y,\nabla z)-\mean{\mathcal{V}(\cdot,\nabla z)}_{B_0}|^2\,dy\right)^{\frac{1}{2}}
\end{split}
\end{align}
for $c=c(n,p,\Lambda)>0$.
\end{proposition}
\begin{proof}
Let $\kappa_5$ and $s$ be as in Proposition \ref{prop:dec}. Since $\mathcal{V}(\cdot,\nabla z)\in L^2(\setR^n)$, $\mathcal{V}(\cdot,F)\in L^2(4B_0\cap\Omega)$ and $|u|^p\omega^p\in L^s(4B_0\cap\Omega)$ by Proposition \ref{prop:SP}, all terms in \eqref{eq:sharpmax} are finite for a.e. $x$. Choose $x\in\setR^n$ and denote
\begin{align}\label{eq:sharpmax1}
I:=\mathcal{M}^{\sharp}_2(\mathcal{V}(\cdot,\nabla z))(x)=\sup_{r>0}\left(\dashint_{B_r(x)}|\mathcal{V}(y,\nabla z)-\mean{\mathcal{V}(\cdot,\nabla z)}_{B_r(x)}|^2\,dy\right)^{\frac{1}{2}}.
\end{align}
We divide the case for $r\in(0,\infty)$ as follows:

\begin{enumerate}
\item[(1)] $J_1:=\{r>0:B_r(x)\cap B_0\cap\Omega=\emptyset\}$
\item[(2)] $J_2:=\{r>0:\frac{4}{\lambda}B_r(x)\subset 4B_0\,\,\text{and}\,\,x\in\overline{\Omega}\}$
\item[(3)] $J_3:=\{r>0:\frac{4}{\lambda}B_r(x)\subset 4B_0\,\,\text{and}\,\,x\not\in\overline{\Omega}\}$
\item[(4)] $J_4:=\{r>0:B_r(x)\cap B_0\cap\Omega\neq\emptyset\,\,\text{and}\,\,\frac{4}{\lambda}B_r(x)\not\subset 4B_0 \}$.
\end{enumerate}
For $k=1,2,3,4$ let us denote
\begin{align}\label{eq:sharpmax2}
I_k:=\sup_{r\in J_k}\dashint_{B_r(x)}|\mathcal{V}(y,\nabla z)-\mean{\mathcal{V}(\cdot,\nabla z)}_{B_r(x)}|\,dy.
\end{align}

We immediately find $I_1=0$ since $z=0$ in $\setR^n\setminus (B_0\cap\Omega)$. For $I_2$, we apply Proposition \ref{prop:dec} with ${B_r}=\lambda^{-1}B_r(x)$ and $\epsilon^2$ instead of $\epsilon$ to have
\begin{align}\label{eq:sharpmax3}
\begin{split}
I_2&\leq\frac{1}{4}I+c\left(\lvert\log\setM|_{\setBMO(4B_0)}+\|\nabla\psi\|_{\infty}+\epsilon\right)\mathcal{M}_{2s}(\mathcal{V}(\cdot,\nabla z))(x)\\
&\quad+c\,C^*(\epsilon^2)R^{-\frac{p}{2}}\left(\mathcal{M}_{s}\left(\chi_{4B_0}|u|^p\omega^p\right)(x)\right)^{\frac{1}{2}}+c\,C^*(\epsilon^2)\mathcal{M}_{2s}(\chi_{4B_0}\mathcal{V}(\cdot,F))(x).
\end{split}
\end{align}
For $I_3$, when $r\in J_1$, then $I_3\equiv 0$. If $r\in J_3\setminus J_1$, then for $x=(x_1,\dots,x_n)$, denote $\tilde{x}=(x_1,\dots,x_n+r)$ and consider $2B_{r}(\tilde{x})(\supset B_r(x))$. Then since $\tilde{x}\in\overline{\Omega}$ and $\frac{2}{\lambda}B_r(\tilde{x})\subset 4B_0$, we can apply Proposition \ref{prop:dec} similarly as above with $B=\lambda^{-1}B_r(\tilde{x})$ and $\epsilon^2$ instead of $\epsilon$ and obtain
\begin{align}\label{eq:sharpmax4}
\begin{split}
I_3&\leq\frac{1}{4}I+c\left(\lvert\log\setM|_{\setBMO(4B_0)}+\|\nabla\psi\|_{\infty}+\epsilon\right)\mathcal{M}_{2s}(\mathcal{V}(\cdot,\nabla z))(x)\\
&\quad+c\,C^*(\epsilon^2)R^{-\frac{p}{2}}\left(\mathcal{M}_{s}\left(\chi_{4B_0}|u|^p\omega^p\right)(x)\right)^{\frac{1}{2}}+c\,C^*(\epsilon^2)\mathcal{M}_{2s}(\chi_{4B_0}\mathcal{V}(\cdot,F))(x).
\end{split}
\end{align}
For $I_4$, since $r\in J_4$ implies $r\geq cR$, and so together with $\support z\subset\overline{B_0}$, we have
\begin{align}\label{eq:sharpmax5}
I_4\leq c\dfrac{R^n}{(R+|x-x_0|)^n}\left(\dashint_{B_0}|\mathcal{V}(y,\nabla z)-\mean{\mathcal{V}(\cdot,\nabla z)}_{B_0}|^2\,dy\right)^{\frac{1}{2}}.
\end{align}
Merging the above estimates, taking the supremum for all $r>0$, and absorbing $\frac{1}{4}I$ in the estimates of $I_2$ and $I_3$ to the left-hand side, the conclusion holds.
\end{proof}

\subsection{Proof of Theorem \ref{thm:nonlin}}
Now, we prove Theorem \ref{thm:nonlin}, the non-linear case. To extract the sharp dependency of $q$, we apply the following global Fefferman-Stein inequality.

\begin{lemma}{\cite[Theorem 20]{BDGN}}\label{thm:sharp1}
Let $q>1$. Then for all $f\in L^q(\setR^n)$ and $g\in L^{q'}(\setR^n)$, we have
\begin{align}\label{eq:sharp1}
\|f\|_{L^q(\setR^n)}\leq c q \|\mathcal{M}^{\sharp}_1 f\|_{L^q(\setR^n)}
\end{align}	
and
\begin{align}\label{eq:sharp1.1}
\|\mathcal{M}_1g\|_{L^{q'}(\setR^n)}\leq c q \|g\|_{L^{q'}(\setR^n)}
\end{align}	
for some $c=c(n)$.
\end{lemma}

We also need the following lemma, which is from \cite{DKS}.

\begin{lemma}\label{lem:rev.holder}
Let $B\subset\setR^n$ be a ball and $g,h:B\rightarrow\setR$ be such that $g,h\in L^1(B)$. Suppose that for some $\theta\in(0,1)$, we have 
\begin{align}\label{eq:rev.holder1}
\dashint_{\tilde{B}}|g|\,dx\leq c_0\left(\dashint_{2\tilde{B}}|g|^{\theta}\,dx\right)^{\frac{1}{\theta}}+\dashint_{2\tilde{B}}|h|\,dx
\end{align}
for any $2\tilde{B}\subset B$. Then for any $\gamma\in(0,1)$, there holds
\begin{align}\label{eq:rev.holder2}
\dashint_{B}|g|\,dx\leq c_1\left(\dashint_{2B}|g|^{\gamma}\,dx\right)^{\frac{1}{\gamma}}+c_1\dashint_{2B}|h|\,dx
\end{align}
for some constant $c_1=c_1(c_0,\gamma,\theta)$. Here, $c_1$ is an increasing function on $c_0$.
\end{lemma}

Now, we prove gradient estimate results for the local boundary case, when there is a priori assumption $u\in W^{1,q}_{\omega}(4B_0\cap\Omega)$.
\begin{proposition}[Local boundary estimate]\label{prop:bdary}
Assume \eqref{eq:Omega} and let $u\in W^{1,q}_{\omega}(\Omega)$ be a weak solution of \eqref{eq:plap4} with $F\in L^{q}_{\omega}(\Omega)$ for $q\in(p,\infty)$. Then there exists $\delta=\delta(n,p,\Lambda)$ such that for any balls $B$ with $x_{B}\in\partial\Omega$, $r_{B}\leq 4R$ and all $q\in[p,\infty)$ with
\begin{align}\label{eq:logM}
\lvert\log\setM|_{\setBMO(8B)}+\|\nabla\psi\|_{\infty}\leq\frac{\delta}{q},
\end{align}
there holds
\begin{align}\label{eq:loc.bdary}
\left(\dashint_{\frac{1}{2}B\cap\Omega}(|\nabla u|^p\omega^p)^\rho\,dx\right)^{\frac{1}{\rho}}\leq \bar{c}\dashint_{4B\cap\Omega}(|\nabla u|\omega)^p\,dx+\bar{c}\left(\dashint_{4B\cap\Omega}(|F|^p\omega^p)^\rho\,dx\right)^{\frac{1}{\rho}}
\end{align}
for some $\bar{c}=\bar{c}(n,\Lambda,q)$ which is continuous on $q$.
\end{proposition}
\begin{proof} 

Define $z=u\eta^{p'}$ as in the previous subsection with $|\nabla u|\omega\in L^{q}(\Omega)$. Let $B_0=B_R(x_0)$ with $x_0\in\partial\Omega$ and $$\rho:=\frac{q}{p}\geq 1.$$ 
We first claim the following type of reverse H\"{o}lder's inequality:
\begin{align}\label{eq:main1}
\begin{split}
\left(\dashint_{\frac{1}{2}B_0\cap\Omega}(|\nabla u|^p\omega^p)^{\rho}\,dx\right)^{\frac{1}{\rho}}&\leq c(q)\left(\dashint_{4B_0\cap\Omega}(|\nabla u|^p\omega^p)^{\theta \rho}\,dx\right)^{\frac{1}{\theta \rho}}\\
&\quad+c(q)\left(\dashint_{4B_0\cap\Omega}(|F|^p\omega^p)^\rho\,dx\right)^{\frac{1}{\rho}}
\end{split}
\end{align}
for some $\theta\in(0,1)$.
If $1\leq\rho\leq s$ where $s$ is defined in Corollary \ref{cor:hi}, then the conclusion directly follows from Corollary \ref{cor:hi}. Hence we only consider the case $\rho> s$. To prevent constants blowing up as $\rho$ close to 1, we change $s$ to $s^2$ so that $1<s<s^2<\rho$ holds.

Let $\epsilon\leq\min\{\frac{\kappa_5}{p},\frac{1}{n}\}$, where $\kappa_5$ is as in Proposition \ref{prop:sharpmax}. Under the assumptions $\lvert\log\setM|_{\setBMO(4B_0)}\leq\epsilon$ and $\|\nabla\psi\|_{\infty}\leq\epsilon$, taking $L^{2\rho}(B_0\cap\Omega)$ norm $\|\cdot\|_{2\rho}$ to Proposition \ref{prop:sharpmax}, we have
\begin{align}\label{eq:main3}
\begin{split}
I:=\|\mathcal{M}^{\sharp}_{2}(\mathcal{V}(\cdot,\nabla z))\|_{2\rho}&\leq c\left(\lvert\log\setM|_{\setBMO(4B_0)}+\|\nabla\psi\|_{\infty}+\epsilon\right)\|\mathcal{M}_{2s}(\mathcal{V}(\cdot,\nabla z))\|_{2\rho}\\
&\quad+c\,C^*(\epsilon^2)R^{-\frac{p}{2}}\|\mathcal{M}_{s}(\chi_{4B_0\cap\Omega}|u|^p\omega^p)^{\frac{1}{2}}\|_{2\rho}\\
&\quad+c\,C^*(\epsilon^2)\|\mathcal{M}_{2s}(\chi_{4B_0\cap\Omega}\mathcal{V}(\cdot,F))\|_{2\rho}\\
&\quad+c\left\|\dfrac{R^n}{(R+|\cdot-x_0|)^n}\right\|_{2\rho}\left(\dashint_{B_0}|\mathcal{V}(x,\nabla z)-\mean{\mathcal{V}(\cdot,\nabla z)}_{B_0}|^2\,dx\right)^{\frac{1}{2}}\\
&=:I_1+I_2+I_3+I_4.
\end{split}
\end{align}
Since $|\nabla u|^p\omega^p\in L^{\rho}(B_0)$, it follows that $\mathcal{V}(\cdot,\nabla z)\in L^{2q}(\setR^n)$ and so $I<\infty$. 
First, using Lemma \ref{thm:sharp1} together with $\mathcal{M}_{2s}(g)=(\mathcal{M}(|g|^{2s}))^{\frac{1}{2s}}$ and $\frac{2\rho}{2s}\geq s\geq 1$, we obtain
\begin{align}\label{eq:main4}
\begin{split}
\|\mathcal{M}_{2s}(\mathcal{V}(\cdot,\nabla z))\|_{2\rho}&\leq c_s\frac{2\rho}{2\rho-1}\|\mathcal{V}(\cdot,\nabla z)\|_{2\rho}\\
&\leq c_s\frac{(2\rho)^2}{2\rho-1}\|\mathcal{M}^{\sharp}_1(\mathcal{V}(\cdot,\nabla z))\|_{2\rho}\leq c_s\frac{(2\rho)^2}{2\rho-1}\|\mathcal{M}^{\sharp}_2(\mathcal{V}(\cdot,\nabla z))\|_{2\rho}.
\end{split}
\end{align}
Thus for $I_1$, one can see that
\begin{align}\label{eq:main5}
I_1\leq c_3q\left(\lvert\log\setM|_{\setBMO(4B_0)}+\|\nabla\psi\|_{\infty}+\epsilon\right)I
\end{align}
for some $c_3=c_3(n,p,\Lambda)$. Here, we choose 
\begin{align*}
\delta=\min\left\{\dfrac{1}{6c_3},\frac{\kappa_5}{p},\frac{1}{n}\right\}\quad\text{and}\quad\epsilon=\dfrac{\delta}{q}
\end{align*}
so that
\begin{align}\label{eq:main6}
I_1\leq\frac{1}{2}I
\end{align}
holds. Then we are able to absorb $I_1$ to $I$. For the remaining term $I_2, I_3$ and $I_4$, by \eqref{eq:sharp1.1} one can see that
\begin{align}\label{eq:main7}
\begin{split}
I_2&\leq c\,C^*(\tfrac{1}{q^2})R^{-\frac{p}{2}}\left\|\mathcal{M}_s[(\chi_{4B_0\cap\Omega}|u|^p\omega^p)^\frac{1}{2}]\right\|_{2\rho}\\
&\leq c\,C^*(\tfrac{1}{q^2})R^{-\frac{p}{2}}\tfrac{\rho}{\rho-s}\left\|\mathcal{M}_1[(\chi_{4B_0\cap\Omega}|u|^p\omega^p)^s]\right\|_{\frac{\rho}{s}}^{\frac{1}{2s}}\\
&\leq c\,C^*(\tfrac{1}{q^2})R^{-\frac{p}{2}}\left\|(\chi_{4B_0\cap\Omega}|u|^p\omega^p)^s\right\|_{\frac{\rho}{s}}^{\frac{1}{2s}}\\
&\leq c\,C^*(\tfrac{1}{q^2})\left(\int_{4B_0\cap\Omega}\left(\dfrac{|u|^p}{R^p}\omega^p\right)^\rho\,dx\right)^{\frac{1}{2\rho}}
\end{split}
\end{align}
with $c=c(n,p,\Lambda)$, and similarly,
\begin{align}\label{eq:main8}
I_3\leq c\, C^*(\tfrac{1}{q^2})\left(\int_{4B_0\cap\Omega}|\mathcal{V}(x,F)|^{2\rho}\,dx\right)^{\frac{1}{2\rho}}.
\end{align}
For $I_4$, if we assume \eqref{eq:Omega}, then $|B_0\cap\Omega|\eqsim|B_0|$ holds. Then together with the fact that $z=|\nabla z|=0$ in $B_0\setminus\Omega$, we have
\begin{align}\label{eq:main9}
I_4\leq c|B_0|^{\frac{1}{2\rho}}\left(\dashint_{B_0\cap\Omega}|\mathcal{V}(x,\nabla z)|^2\,dx\right)^{\frac{1}{2}}.
\end{align}
On the other hand, by Lemma \ref{thm:sharp1} there holds
\begin{align}\label{eq:main10}
I=\|\mathcal{M}^{\sharp}_2\mathcal{V}(\cdot,\nabla z)\|_{2\rho}\geq c\|\mathcal{M}^{\sharp}_1\mathcal{V}(\cdot,\nabla z)\|_{2\rho}\geq\dfrac{c}{q}\|\mathcal{V}(\cdot,\nabla z)\|_{2\rho}.
\end{align}

Summing up, we have
\begin{align}\label{eq:main11}
\begin{split}
\|\mathcal{V}(\cdot,\nabla z)\|_{2\rho}&\leq cqC^*(\tfrac{1}{q^2})\left(\int_{4B_0\cap\Omega}\left(\dfrac{|u|^p}{R^p}\omega^p\right)^\rho\,dx\right)^{\frac{1}{2\rho}}\\
&\quad+cqC^*(\tfrac{1}{q^2})\left(\int_{4B_0\cap\Omega}|\mathcal{V}(x,F)|^{2\rho}\,dx\right)^{\frac{1}{2\rho}}\\
&\quad+cq|B_0|^{\frac{1}{2\rho}}\left(\dashint_{B_0\cap\Omega}|\mathcal{V}(x,\nabla z)|^2\,dx\right)^{\frac{1}{2}}.
\end{split}
\end{align}
If we assume \eqref{eq:Omega}, then $|B_0\cap\Omega|\eqsim|B_0|$ and the above estimate implies
\begin{align}\label{eq:main12}
\begin{split}
\left(\dashint_{B_0\cap\Omega}(|\nabla z|^p\omega^p)^{\rho}\,dx\right)^{\frac{1}{\rho}}&\leq c(qC^*(\tfrac{1}{q^2}))^2\left(\dashint_{4B_0\cap\Omega}\left(\dfrac{|u|^p}{R^p}\omega^p\right)^{\rho}\,dx\right)^{\frac{1}{\rho}}\\
&\quad+c(qC^*(\tfrac{1}{q^2}))^2\left(\dashint_{4B_0\cap\Omega}(|F|^p\omega^p)^{\rho}\,dx\right)^{\frac{1}{\rho}}\\
&\quad+cq\dashint_{B_0\cap\Omega}|\mathcal{V}(x,\nabla z)|^2\,dx.
\end{split}
\end{align}
Since $z=u\eta^{p'}$ as in \eqref{eq:z}, it follows that
\begin{align}\label{eq:main13}
\begin{split}
\left(\dashint_{\frac{1}{2}B_0\cap\Omega}(|\nabla u|^p\omega^p)^{\rho}\,dx\right)^{\frac{1}{\rho}}&\leq c(q)\left(\dashint_{4B_0\cap\Omega}\left(\dfrac{|u|^p}{R^p}\omega^p\right)^{\rho}\,dx\right)^{\frac{1}{\rho}}\\
&\quad+c(q)\left(\dashint_{4B_0\cap\Omega}(|F|^p\omega^p)^{\rho}\,dx\right)^{\frac{1}{\rho}}+c(q)\dashint_{B_0\cap\Omega}|\nabla u|^p\omega^p\,dx,
\end{split}
\end{align}
where $c(q)\eqsim (qC^*(\frac{1}{q^2}))^2$ which is a continuous and increasing function on $q$. Then since $\lvert\log\setM|_{\setBMO(4B_0)}\leq\kappa_5$ and \eqref{eq:Omega} holds, together with Lemma \ref{lem:Mu.weight}, we can apply Proposition \ref{prop:SP}. Consequently, we have \eqref{eq:main1}.

Now, using \eqref{eq:main1}, we next claim that
\begin{align}\label{eq:main14.1}
\begin{split}
\left(\dashint_{\frac{1}{2}B}(\chi_{\Omega}|\nabla u|^p\omega^p)^{\rho}\,dx\right)^{\frac{1}{\rho}}&\leq c(q)\left(\dashint_{4B}(\chi_{\Omega}|\nabla u|^p\omega^p)^{\theta \rho}\,dx\right)^{\frac{1}{\theta \rho}}\\
&\quad+c(q)\left(\dashint_{4B}(\chi_{\Omega}|F|^p\omega^p)^\rho\,dx\right)^{\frac{1}{\rho}}
\end{split}
\end{align}
for all $B\subset\setR^n$. Indeed, if $4B\subset\Omega$, then we employ \cite[Proposition 22]{BDGN} to obtain \eqref{eq:main14.1}. If $4B\subset(\setR^n\setminus\Omega)$, then \eqref{eq:main14.1} becomes trivial since $\chi_{\Omega}=0$ on $\setR^n\setminus\Omega$. Finally, if $4B\not\subset\Omega$ and $4B\not\subset(\setR^n\setminus\Omega)$, we use the similar argument of \textbf{Step 3} in the proof of Proposition \ref{prop:regy} and so we have \eqref{eq:main14.1}.

Now, Lemma \ref{lem:rev.holder} gives us that
\begin{align}\label{eq:main15}
\left(\dashint_{\frac{1}{2}B_0\cap\Omega}(|\nabla u|^p\omega^p)^\rho\,dx\right)^{\frac{1}{\rho}}\leq \bar{c}(q)\left(\dashint_{4B_0\cap\Omega}|\nabla u|\omega\,dx\right)^p+\bar{c}(q)\left(\dashint_{4B_0\cap\Omega}(|F|^p\omega^p)^\rho\,dx\right)^{\frac{1}{\rho}},
\end{align}
where $\bar{c}(q)$ is still a continuous and increasing function on $q$. This proves \eqref{eq:loc.bdary}.
\end{proof}

\begin{proof}[Proof of Theorem \ref{thm:nonlin}]
Applying the argument of the proof of \cite[Theorem 2]{BDGN}, we can eliminate the assumption $u\in W^{1,q}_{\omega}(4B_0\cap\Omega)$ in the statement of Proposition \ref{prop:bdary}. Note that here we used the fact that $\bar{c}(q)$ is continuous and increasing function on $q$. Now, by considering Proposition \ref{prop:bdary} for the boundary case and \cite[Theorem 2]{BDGN} for the interior case, using the covering argument, together with the assumption 
\begin{align}\label{eq:small3}
\log\setM\,\,\,\text{is}\,\,\,\left(\dfrac{\delta}{q},R\right)\text{--vanishing and }\Omega\,\,\text{is}\,\,\left(\dfrac{\delta}{q},R\right)\text{--Lipschitz},
\end{align} 
we get 
\begin{align}\label{eq:main16}
\int_{\Omega}(|\nabla u|^p\omega^p)^{\rho}\,dx\leq c^*\left(\int_{\Omega}|\nabla u|^p\omega^p\,dx\right)^{\frac{\rho}{p}}+c^*\int_{\Omega}(|F|^p\omega^p)^{\rho}\,dx,
\end{align}
where $c^*=c^*(n,p,\Lambda,\Omega,R,q)$. Then the standard energy estimate as in \eqref{eq:energy} and H\"{o}lder's inequality imply \eqref{eq:nonlin}.
\end{proof}

\subsection{Proof of Theorem \ref{thm:lin}} Here we prove Theorem \ref{thm:lin} using the duality argument.
\begin{proof}[Proof of Theorem \ref{thm:lin}]
  We only show in the case $1<q<2$, since the case $q\geq 2$ follows from Theorem \ref{thm:nonlin} with $p=2$. As the previous argument, we first prove the local boundary case, and then employ the result of \cite[Theorem 1]{BDGN} as the interior case to use the standard covering argument. Recall that
  \begin{align}\label{eq:lin1}
    -\divergence(\mathbb{A}(x)\nabla u)=-\divergence(\mathbb{A}(x)F)
  \end{align}
  and that $B_0=B_R(x_0)$ with $x_0\in\partial\Omega$. Define $H\in L^{q'}_{\omega}(B_0)$ with the following property:
  \begin{align}\label{eq:lin2}
    \left(\dashint_{2B_0}(|H|\omega)^{q'}\,dx\right)^{\frac{1}{q'}}\leq 1.
  \end{align}
  Let $aB_0\cap\Omega:=B_{aR}(x_0)\cap\Omega$ for $a>0$ and $w\in W^{1,2}_{0,\omega}(2B_0\cap\Omega)$ be the weak solution of 
  \begin{align}\label{eq:lin3}
    \begin{split}
      -\divergence(\mathbb{A}(x)\nabla w)&=-\divergence(\mathbb{A}(x)\chi_{2B_0}H)\quad\text{in }4B_0\cap\Omega,\\
      w&=0\quad\quad\quad\quad\quad\quad\quad\quad\,\,\text{on }\partial(4B_0\cap\Omega).
    \end{split}
  \end{align}
  Under the assumption that
  \begin{align}\label{eq:lin3.1}
    \lvert\log\setM|_{\setBMO(4B_0)}\leq\delta\left(1-\frac{1}{q}\right)\quad\text{and}\quad \|\nabla\psi\|_{\infty}\leq\delta\left(1-\frac{1}{q}\right),
  \end{align}
  by \eqref{eq:main15} with the exponent $q'\geq 2$ and H\"{o}lder's inequality, it follows that
  \begin{align}\label{eq:lin4}
    \begin{split}
      \left(\dashint_{2B_0\cap\Omega}(|\nabla w|\omega)^{q'}\,dx\right)^{\frac{1}{q'}}\leq c\left(\dashint_{4B_0\cap\Omega}(|\nabla w|\omega)^2\,dx\right)^{\frac{1}{2}}+c\left(\dashint_{2B_0\cap\Omega}(|H|\omega)^{q'}\,dx\right)^{\frac{1}{q'}}.
    \end{split}
  \end{align}
  Here, testing $w$ itself in \eqref{eq:lin3}, we have
  \begin{align}\label{eq:lin5}
    \begin{split}
      \dashint_{4B_0\cap\Omega}(|\nabla w|\omega)^2\,dx\leq c\dashint_{2B_0\cap\Omega}(|H|\omega)^2\,dx\leq c\left(\dashint_{2B_0\cap\Omega}(|H|\omega)^{q'}\,dx\right)^{\frac{2}{q'}}
    \end{split}
  \end{align}
  and so there holds
  \begin{align}\label{eq:lin6}
    \left(\dashint_{2B_0\cap\Omega}(|\nabla w|\omega)^{q'}\,dx\right)^{\frac{1}{q'}}\leq c\left(\dashint_{2B_0\cap\Omega}(|H|\omega)^{q'}\,dx\right)^{\frac{1}{q'}}\leq c.
  \end{align}
  Let $\eta\in C^{\infty}_0(2B_0)$ be a smooth cut-off function with $\chi_{B_0}\leq\eta\leq\chi_{2B_0}$ and $\|\nabla\eta\|_{\infty}\leq c/R$. From \eqref{eq:lin3}, we have
  \begin{align}\label{eq:lin7}
    \begin{split}
      I&:=\dashint_{2B_0\cap\Omega}\mathbb{A}(x)\nabla (\eta^2u)\cdot H\,dx\\
      &=\dashint_{2B_0\cap\Omega}\mathbb{A}(x)\nabla(\eta^2u)\cdot\nabla w\,dx\\
      &=\dashint_{2B_0\cap\Omega}\mathbb{A}(x)\nabla u\cdot\nabla(\eta^2w)\,dx\\
      &\quad+\dashint_{2B_0\cap\Omega}\mathbb{A}(x)u\nabla(\eta^2)\cdot\nabla w\,dx-\dashint_{2B_0\cap\Omega}\mathbb{A}(x)w\nabla u\cdot\nabla (\eta^2)\,dx\\
      &=:I_1+I_2+I_3.
    \end{split}
  \end{align}
  To estimate $I_1$, using the equation for $u$ in \eqref{eq:plap4}, there holds
  \begin{align}\label{eq:lin8}
    \begin{split}
      I_1&=\dashint_{2B_0\cap\Omega}\mathbb{A}(x)F\cdot\nabla(\eta^2w)\,dx\\
      &\leq c\dashint_{2B_0\cap\Omega}\omega^2|F||\nabla(\eta^2w)|\,dx\\
      &\leq c\left(\dashint_{2B_0\cap\Omega}(\omega|F|)^{q}\,dx\right)^{\frac{1}{q}}\left(\dashint_{2B_0\cap\Omega}(\omega|\nabla(\eta^2w)|)^{q'}\,dx\right)^{\frac{1}{q'}}.
    \end{split}
  \end{align}
  With the help of triangle inequality and Proposition \ref{prop:SP}, we have
  \begin{align}\label{eq:lin9}
    |I_1|\leq\left(\dashint_{2B_0\cap\Omega}(\omega|F|)^{q}\,dx\right)^{\frac{1}{q}}\left(\dashint_{2B_0\cap\Omega}(\omega|\nabla w|)^{q'}\,dx\right)^{\frac{1}{q'}}.
  \end{align}
  For $I_2$, by H\"{o}lder's inequality and Proposition \ref{prop:SP}, we have
  \begin{align}\label{eq:lin10}
    \begin{split}
      |I_2|&\leq c\left(\dashint_{2B_0\cap\Omega}\left(\omega\dfrac{|u|}{R}\right)^{q}\,dx\right)^{\frac{1}{q}}\left(\dashint_{2B_0\cap\Omega}(\omega|\nabla w|)^{q'}\,dx\right)^{\frac{1}{q'}}\\
      &\leq c\left(\dashint_{2B_0\cap\Omega}(\omega|\nabla u|)^{\theta p}\,dx\right)^{\frac{1}{\theta p}}\left(\dashint_{2B_0\cap\Omega}(\omega|\nabla w|)^{q'}\,dx\right)^{\frac{1}{q'}}
    \end{split}
  \end{align}
  for some $\theta\in(\frac{1}{q},1)$. Similarly, for $I_3$, there holds
  \begin{align}\label{eq:lin12}
    \begin{split}
      |I_3|&\leq c\left(\dashint_{2B_0\cap\Omega}(\omega|\nabla u|)^{\theta_2 p}\,dx\right)^{\frac{1}{\theta_2 p}}\left(\dashint_{2B_0\cap\Omega}\left(\omega\dfrac{|w|}{R}\right)^{(\theta_2 q)'}\,dx\right)^{\frac{1}{(\theta_2 q)'}}\\
      &\leq c\left(\dashint_{2B_0\cap\Omega}(\omega|\nabla u|)^{\theta_2q}\,dx\right)^{\frac{1}{\theta_2q}}\left(\dashint_{2B_0\cap\Omega}(\omega|\nabla w|)^{q'}\,dx\right)^{\frac{1}{q'}}.
    \end{split}
  \end{align}
  Now, without loss of generality we assume $\theta=\theta_2$. Consequently, with \eqref{eq:lin6} we have
  \begin{align}\label{eq:lin14}
    \begin{split}
      |I|&\leq c\left[\left(\dashint_{2B_0\cap\Omega}(\omega|\nabla u|)^{\theta q}\,dx\right)^{\frac{1}{\theta q}}+\left(\dashint_{2B_0\cap\Omega}(\omega|F|)^{q}\,dx\right)^{\frac{1}{q}}\right]\left(\dashint_{2B_0\cap\Omega}(\omega|\nabla w|)^{q'}\,dx\right)^{\frac{1}{q'}}\\
      &\leq c\left(\dashint_{2B_0\cap\Omega}(\omega|\nabla u|)^{\theta q}\,dx\right)^{\frac{1}{\theta q}}+c\left(\dashint_{2B_0\cap\Omega}(\omega|F|)^{q}\,dx\right)^{\frac{1}{q}}.
    \end{split}
  \end{align}

Since $H$ was an arbitrary function with \eqref{eq:lin2} and $(L^{q'}_{\omega})^*=L^{q}_{\omega^{-1}}$ holds, we obtain
\begin{align}\label{eq:lin16}
\begin{split}
&\left(\dashint_{2B_0\cap\Omega}\left(|\mathbb{A}\nabla(\eta^2u)|\omega^{-1}\right)^{q}\,dx\right)^{\frac{1}{q}}\\
&\quad\quad\leq c\left(\dashint_{2B_0\cap\Omega}(\omega|\nabla u|)^{\theta q}\,dx\right)^{\frac{1}{\theta q}}+c\left(\dashint_{2B_0\cap\Omega}(\omega|F|)^{q}\,dx\right)^{\frac{1}{q}}. 
\end{split}
\end{align}
Since $\mathbb{A}\nabla (\eta^2u)=\mathbb{A}\nabla u$ on $2B_0\cap\Omega$ and $|\mathbb{A}\nabla u|\eqsim\omega^2|\nabla u|$ hold, we conclude
\begin{align}\label{eq:lin17}
\left(\dashint_{B_0\cap\Omega}(\omega|\nabla u|)^{q}\,dx\right)^{\frac{1}{q}}\leq  c\left(\dashint_{2B_0\cap\Omega}(\omega|\nabla u|)^{\theta q}\,dx\right)^{\frac{1}{\theta q}}+c\left(\dashint_{2B_0\cap\Omega}(\omega|F|)^{q}\,dx\right)^{\frac{1}{q}},
\end{align}
which is analogous to \eqref{eq:main1}. Now, by applying the argument of the proof of Proposition \ref{prop:bdary}, we can change the exponent $\theta q$ to 1. Then similar to the proof of Theorem \ref{thm:nonlin}, using the covering argument, together with the assumptions \begin{align}\label{eq:small4}
\log\mathbb{A}\,\,\,\text{is}\,\,\,\left(\delta\min\left\{\frac{1}{q},1-\frac{1}{q}\right\},R\right)\text{--vanishing and}
  \end{align}
  \begin{align}\label{eq:Omegasmall4}		
    \Omega\,\,\text{is}\,\,\left(\delta\min\left\{\frac{1}{q},1-\frac{1}{q}\right\},R\right)\text{--Lipschitz},
  \end{align}  
  we get \eqref{eq:lin}. This proves Theorem \ref{thm:lin}.
\end{proof}

\section{Sharpness and Smallness Condition}\label{sec:4}

In this section we discuss the sharpness of our smallness condition and compare it to other type of conditions as found in \cite{CMP1,CMP2}.

\subsection{Sharpness of the Linear Dependence}
\label{ssec:sharpness}

We have shown in Theorem~\ref{thm:lin} and Theorem~\ref{thm:nonlin} that reciprocal of the exponent~$q$ of higher integrability is linearly connected to the smallness condition on~$\log\bbA$ and $\partial \Omega$. In this section we show that this linear dependence is the best possible.  It has been shown already in~\cite[Section~4]{BDGN} that the smallness on~$\log \bbA$ is necessary by means of analyzing the counterexample introduced by Meyers \cite{Me1}. Therefore, we concentrate in this article on the sharpness of the condition on~$\Omega$. Since the effect already occurs in the unweighted case, we assume that $\setM=\identity$. We provide a two dimensional example, but the principle generalizes to higher dimensions as well.
\begin{example}
  For $n=2$ and $\epsilon\in(0,1)$, we consider the following type of the domain:
  \begin{align*}
    \Omega=\{x=(x_1,x_2)\in B_1(0):x_2>-\epsilon|x_1|\}.
  \end{align*}
  Then $\Omega$ is $(\epsilon,1)$-Lipschitz. Moreover, the assigned Lipschitz map for the origin is $\psi:\setR\rightarrow\setR$ such that $\psi(x_1)=-\epsilon|x_1|$ and so $\|\nabla\psi\|_{\infty}\leq\epsilon$.

  Now, for $\alpha := \frac{\pi/2}{\pi/2+\tan^{-1} \epsilon}$ we define in polar coordinates $x=r(\cos \phi,\sin\phi)$
  \begin{align*}
    u(x)=u(r,\phi)=\cos\big(\alpha (\phi-\tfrac{\pi}{2}) \big)r^\alpha.
  \end{align*}
  Then $u$ is a solution of the equation
  \begin{align*}
    \begin{split}
      \Delta u&=0\quad\text{in }\Omega,\\
      u&=0\quad\text{on }\{(x_1,x_2)\in\setR^2:x_2=-\epsilon|x_1|\}\cap B_1(0).
    \end{split}
  \end{align*}
  Indeed, we have
  \begin{align*}
    \begin{split}
      \Delta u&=\dfrac{\partial^2 u}{\partial r^2}+\frac{1}{r}\dfrac{\partial u}{\partial r}+\frac{1}{r^2}\dfrac{\partial^2u}{\partial\phi^2}
      \\
      &= \Big( \alpha(\alpha-1) r^{\alpha-2} +\frac{\alpha}{r} r^{\alpha-1}-\frac{\alpha^2}{r^2} \Big) \cos\big(\alpha (\phi-\tfrac{\pi}{2}) \big)r^\alpha=0.
    \end{split}
  \end{align*}
  Moreover, with $e_r := (\cos \phi,\sin\phi)$ and $e_\phi := (-\sin \phi, \cos\phi)$ that
  \begin{align*}
  \begin{split}
  |\nabla u|&= \biggabs{\frac{\partial u}{\partial r} e_r + \frac{\partial u}{r \partial \phi} e_\phi}\\
  &=\bigabs{ \alpha \cos\big(\alpha (\phi-\tfrac{\pi}{2}) \big) r^{\alpha-1} e_r - \alpha \sin\big(\alpha (\phi-\tfrac{\pi}{2}) \big) r^{\alpha-1} e_\phi} = \alpha r^{\alpha-1}.
  \end{split}
  \end{align*}
  Assume that $q > 2$. Then~$\nabla u \in L^{q,\infty}(B_1(0))$ (Lorentz space or weak Lebesgue space) is equivalent to $(\alpha-1)q \geq -2$. This simplifies to
  \begin{align*}
    \nabla u \in L^{q,\infty}(B_1(0)) \quad \Leftrightarrow \quad \tan^{-1} \epsilon \leq \frac{\pi}{q-2}.
  \end{align*}
  Thus,
  \begin{align*}
    \nabla u \in L^q(B_1(0)) \quad \Leftrightarrow \quad \tan^{-1} \epsilon < \frac{\pi}{q-2}.
  \end{align*}
  Note that for small~$\epsilon$, we have $\tan^{-1} \epsilon \approx \epsilon$.
  This implies that the smallness assumptions \eqref{eq:Omegasmall} and \eqref{eq:small2} are optimal.
\end{example}

\subsection{Comparison to the Condition of Cao, Mengesha and Phan}
\label{ssec:comparison-CMP}

Let us assume that $\bbA\,:\, \setR^n \to \Rnnpos$ with bounded condition number with $\abs{\bbA} \abs{\bbA^{-1}}\leq \Lambda^2$.  In this section we compare our smallness condition on the weight in terms of~$\log \bbM$ with the smallness condition on~$\bbM$ from Cao, Mengesha and Phan in~\cite{CMP1,CMP2}.  In~\cite{CMP1}, they introduced the quantity\footnote{Cao, Mengesha and Phan do not use $\mu(x) := \abs{\bbA(x)}$, but treat $\mu$ as an independent function that satisfies the equivalence $\mu(x) \eqsim \abs{\bbA(x)}$. However, choosing $\mu(x) := \abs{\bbA(x)}$ is always an equivalent valid choice.}
\begin{align}
  \label{eq:cond1}
  \abs{\mathbb{A}}_{\setBMO^2_{\mu}}:=\sup_B \bigg( \dfrac{1}{\mu(B)}\int_{B}|\mathbb{A}(x)-\mean{\mathbb{A}}_B|^2\mu^{-1}(x)\,dx \bigg)^{\frac 12}
\end{align}
in order to measure the oscillations of~$\bbA$, where $\mu(x) := \abs{\bbA(x)}$ and the supremum is taken over all balls.
In addition to the smallness conditions Cao, Mengesha and Phan assume that $\mu := \abs{\bbA} \in \mathcal{A}_2$.

In~\cite{CMP2} they use the simpler quantity
\begin{align}
  \label{eq:cond2}
  \abs{\mathbb{A}}_{\setBMO_{\mu}} :=\sup_B\dfrac{1}{\mu(B)} \int_{B}|\mathbb{A}(x)-\mean{\mathbb{A}}_B| \,dx.
\end{align}
Note that by H\"older's inequality
\begin{align*}
  \abs{\mathbb{A}}_{\setBMO_{\mu}} &\leq   \abs{\mathbb{A}}_{\setBMO^2_{\mu}}.
\end{align*}
In contrast our measure of oscillations is
\begin{align}
  \label{eq:cond-we}
  \abs{\log \bbA}_{\setBMO} = \sup_B \dashint_B \abs{\log \bbA - \mean{\log \bbA}_B}\,dx.
\end{align}
Due to $1 \leq \abs{\bbA} \abs{\bbA^{-1}}\leq \Lambda^2$  we have
\begin{align}
  \label{eq:mu-A-muck}
  \mean{\mu}_B \mean{\mu^{-1}}_B \leq \mean{\abs{\bbA}}_B \mean{\abs{\bbA^{-1}}}_B \leq \Lambda^2
  \mean{\mu}_B \mean{\mu^{-1}}_B.
\end{align}

The following lemma shows that our smallness condition on $\abs{\log \bbA}_{\setBMO}$ is weaker than the smallness condition on $\abs{\bbA}_{\setBMO^2_\mu}$ combined with the $\mathcal{A}_2$-condition.
\begin{lemma}
  \label{lem:log-weakest}
  Let $\bbA\,:\, \setR^n\to \Rnnpos$ be a weight with $\abs{\bbA} \abs{\bbA^{-1}} \leq \Lambda^2$ and $\mu := \abs{\bbA} \in \mathcal{A}_2$. Then for all balls~$B \subset \setR^n$ there holds
  \begin{align*}
    \dashint_B \abs{\log \bbA - \mean{\log \bbA}_B}\,dx
    &\leq 4\, \mean{\abs{\bbA}}_B \mean{\abs{\bbA^{-1}}}_B \bigg( \dfrac{1}{\mu(B)}\int_{B}|\mathbb{A}(x)-\mean{\mathbb{A}}_B|^2\mu^{-1}(x)\,dx \bigg)^{\frac 12}.
  \end{align*}
  Moreover, $\abs{\log \bbA}_{\setBMO} \leq 4\, \Lambda^2 [\mu]_2 \abs{\bbA}_{\setBMO^2_\mu}$.
\end{lemma}
\begin{proof}
  We begin with
  \begin{align*}
    \dashint_B \abs{\log \bbA - \mean{\log \bbA}_B}\,dx
    &\leq 2 \inf_{\bbA_0 \in \Rnnpos} \dashint_B \abs{\log \bbA - \log \bbA_0}\,dx
      \leq 2\dashint_B \abs{\log \bbA - \log (\mean{\bbA}_B)}\,dx.
  \end{align*}
  It has been shown e.g. in~\cite[Example~1]{Gha21} that for all $\bbG, \bbH \in \Rnnpos$ there holds
  \begin{align*}
    \abs{\log \bbG-\log \bbH}
    & \leq  \max \set{ \abs{\bbG^{-1}}, \abs{\bbH^{-1}}}
      \abs{\bbG-\bbH}.
  \end{align*}
  This implies
  \begin{align*}
    \dashint_B \abs{\log \bbA - \mean{\log \bbA}_B}\,dx
    &\leq 2\dashint_B \max \set{\abs{\bbA^{-1}}, \abs{\mean{\bbA}_B^{-1}}} \abs{\bbA -\mean{\bbA}_B}\,dx
      \\
    &\leq 2\dashint_B \abs{\bbA^{-1}} \abs{\bbA -\mean{\bbA}_B}\,dx + 2 \abs{\mean{\bbA}_B^{-1}}\dashint_B   \abs{\bbA -\mean{\bbA}_B}\,dx =: \textrm{I} + \textrm{II}.
  \end{align*}
  By H\"older's inequality we obtain
  \begin{align*}
    \textrm{I} = \dashint_B \abs{\bbA^{-1}} \abs{\bbA -\mean{\bbA}_B}\,dx
    &\leq       \bigg(\dashint_B\abs{\bbA^{-1}}\,dx \bigg)^{\frac 12}
      \bigg(\dashint_B\abs{\bbA -\mean{\bbA}_B}^2 \abs{\bbA}^{-1}\,dx \bigg)^{\frac 12}
    \\
    &\leq  \mean{\abs{\bbA}}_B \mean{\abs{\bbA^{-1}}}_B 
     \bigg(\frac{1}{\mean{\abs{\bbA}}_B} \int_B\abs{\bbA -\mean{\bbA}_B}^2 \abs{\bbA}^{-1}\,dx \bigg)^{\frac 12}.
  \end{align*}
  On the other hand by H\"older's inequality
  \begin{align*}
    \textrm{I} \leq \abs{\mean{\bbA}_B^{-1}} \dashint_B \abs{\bbA -\mean{\bbA}_B}\,dx
    &\leq \abs{\mean{\bbA}_B^{-1}} \abs{\mean{\bbA}_B}
    \bigg( \dashint_B  \abs{\bbA -\mean{\bbA}_B}^2 \abs{\bbA}^{-1} \,dx \bigg)^{\frac 12}.
  \end{align*}
  Due to \cite[Exercise~1.5.10]{BhaPositive07} the mapping $\bbA \mapsto \bbA^{-1}$ is convex on~$\Rnnpos$. Thus, by Jensen's inequality $0 < \mean{\bbA}_B^{-1} \leq \mean{\bbA^{-1}}_B$ and as a consequence $
  \bigabs{\mean{\bbA}_B^{-1}} \leq \bigabs{\mean{\bbA^{-1}}_B} \leq  \mean{\abs{\bbA^{-1}}}_B$.
  Using this fact,  we obtain for~$\textrm{II}$ the same estimate as for~$\textrm{I}$. Combining all estimates proves the first claim. Taking the supremum over all balls~$B$ and using~\eqref{eq:mu-A-muck} proves the second claim.
\end{proof}

On the other hand we will show now that if $\abs{\log\bbA}_{\setBMO}$ is small enough, then it controls $\abs{\bbA}_{\setBMO^2_{\mu}}$ in a linear way.
\begin{lemma}
  Let $\bbA\,:\, \setR^n\to \Rnnpos$ be a weight such $\abs{\bbA} \abs{\bbA^{-1}} \leq \Lambda^2$. Then there exists $\delta=\delta(n,\Lambda)>0$ such that the following holds: If $\abs{\log \bbA}_{\setBMO} \leq \delta$, then for all balls~$B\subset \setR^n$
  \begin{align*}
      \bigg( \frac{\abs{B}}{\mu(B)}\dashint_{B}|\mathbb{A}(x)-\mean{\mathbb{A}}_B|^2\mu^{-1}(x)\,dx \bigg)^{\frac 12}  \leq c(n,\Lambda) \abs{\log\mathbb{A}}_{\setBMO(B)}.
  \end{align*}
  In particular, $\abs{\bbA}_{\setBMO^2_\mu(B)} \leq c(n,\Lambda)\abs{\log\mathbb{A}}_{\setBMO(B)}$.
\end{lemma}
\begin{proof}
  Let $\abs{\log\mathbb{A}}_{\setBMO(B)}\leq\delta$. We choose~$\delta>0$ so small such that we can apply Lemma~\ref{lem:BMOM} (for $t=4$) and Lemma~\ref{lem:Mu.weight} (for~$\gamma=2$). 
  By Lemma~\ref{lem:Mu.weight}~\ref{itm:Mu.weight-2} we obtain $\mu = \abs{\bbA} \in \mathcal{A}_2$ with $[\mu]_{\mathcal{A}_2} = [\abs{\bbA}]_{\mathcal{A}_2} \leq 16$. Thus, with~\eqref{eq:mu-A-muck} we have $\mean{\abs{\bbA}}_B \mean{\abs{\bbA^{-1}}}_B \leq 16\,\Lambda^2$.
  Recall that $\mu(B) = \int_B \mu\,dx = \int_B \abs{\bbA}\,dx = \abs{B} \mean{\abs{\bbA}}_B$.
  For all $\bbA_0 \in \Rnnpos$ we estimate with Jensen's inequality in the third step
  \begin{align*}
    \lefteqn{ \bigg(\frac{1}{\mu(B)}\int_{B}|\mathbb{A}(x)-\mean{\mathbb{A}}_B|^2\mu^{-1}(x)\,dx \bigg)^{\frac 12}
    = \bigg(\frac{1}{\mean{\abs{\bbA}}_B}\dashint_{B}|\mathbb{A}(x)-\mean{\mathbb{A}}_B|^2 \abs{\bbA(x)}^{-1}\,dx \bigg)^{\frac 12}
    } \qquad
    \\
    &\leq
      \bigg( \frac{1}{\mean{\abs{\bbA}}_B}\dashint_{B}|\mathbb{A}(x)-\bbA_0|^2 \abs{\bbA(x)}^{-1}\,dx\bigg)^{\frac 12} +
      \bigg(\frac{1}{\mean{\abs{\bbA}}_B}\dashint_{B}|\mean{\bbA-\bbA_0}_B|^2 \abs{\bbA(x)}^{-1}\,dx \bigg)^{\frac 12}=: \textrm{I} + \textrm{II}.
  \end{align*}
  With Jensen's inequality, H\"older's inequality and $\mu^{-1} = \abs{\bbA}^{-1} \leq \abs{\bbA^{-1}}$ we obtain
  \begin{align*}
    \textrm{II}
    &\leq
      \dashint_B|\mathbb{A}(x)-\bbA_0| \,dx \, \bigg( \frac{\mean{\abs{\bbA}^{-1}}_B}{\mean{\abs{\bbA}}_B} \bigg)^{\frac 12}
    \\
    &\leq
      \bigg( \dashint_B|\mathbb{A}(x)-\bbA_0|^2 \abs{\bbA(x)}^{-1}\,dx\bigg)^{\frac 12} \mean{\abs{\bbA}}_B^{\frac 12} \bigg( \frac{\mean{\abs{\bbA}^{-1}}_B}{\mean{\abs{\bbA}}_B} \bigg)^{\frac 12}
    \\
    &= \big(\mean{\abs{\bbA}}_B \mean{\abs{\bbA^{-1}}}_B \big)^{\frac 12}\,       \bigg( \frac{1}{\mean{\abs{\bbA}}_B}\dashint_{B}|\mathbb{A}(x)-\bbA_0|^2 \abs{\bbA(x)}^{-1}\,dx\bigg)^{\frac 12}
    \\
    &= \big( \mean{\abs{\bbA}}_B \mean{\abs{\bbA^{-1}}}_B\big)^{\frac 12} \, \textrm{I} \leq\, 4\,\Lambda\, \textrm{I}.
  \end{align*}
  Overall, we obtain
  \begin{align*}
    \lefteqn{\bigg(\frac{1}{\mu(B)}\int_{B}|\mathbb{A}(x)-\mean{\mathbb{A}}_B|^2 \mu^{-1}(x)\,dx \bigg)^{\frac 12}} \qquad
    &
    \\
    &\leq  (1+4\, \Lambda) \underbrace{\inf_{\bbA_0 \in \Rnnpos} \bigg(  \frac{1}{\mean{\abs{\bbA}}_B} \dashint_{B}|\mathbb{A}(x)-\bbA_0|^2 \abs{\bbA(x)}^{-1}\,dx\bigg)^{\frac 12}}_{ =: \textrm{III}}.
  \end{align*}
  Now, choosing $\bbA_0= \logmean{\bbA}_B$ and using H\"older's inequality we obtain
  \begin{align*}
    \textrm{III}
    &\leq
      \bigg( \frac{1}{\mean{\abs{\mathbb{A}}}_B} \dashint_{B}\mathbb{A}(x)-\mean{\mathbb{A}}^{\log}_B|^2 \abs{\bbA(x)}^{-1}\,dx \bigg)^{\frac 12}
    \\
    &\leq 
       \frac{\abs{\logmean{\bbA}_B} \mean{\abs{\bbA}^{-2}}_B^{\frac 14}}{\mean{\abs{\mathbb{A}}}_B^{\frac 12}} \bigg( \dashint_{B} \frac{|\mathbb{A}(x)-\mean{\mathbb{A}}^{\log}_B|^4}{\abs{\logmean{\bbA}_B}^4} \,dx \bigg)^{\frac 14} .
  \end{align*}
  With Lemma~\ref{lem:Mu.weight}~\ref{itm:Mu.weight-1} (with $-\gamma=-2$), \eqref{eq:aux-mean1} and \eqref{eq:aux-mean2} we obtain
  \begin{align*}
    \frac{\abs{\logmean{\bbA}_B} \mean{\abs{\bbA}^{-2}}_B^{\frac 14}}{\mean{\abs{\mathbb{A}}}_B^{\frac 12}}
    &\leq     \frac{2\,\abs{\logmean{\bbA}_B} }{\mean{\abs{\mathbb{A}}}_B^{\frac 12} (\logmean{\abs{\bbA}}_B)^{\frac 12} } \leq 2.
  \end{align*}
  This and Lemma~\ref{lem:BMOM} (with $t=2$) gives
  \begin{align*}
    \textrm{III}
    &\leq 2\,  c(n,\Lambda)\, \lvert\log\mathbb{A}|_{\setBMO(B)}.
  \end{align*}
  Collecting the estimates proves the claim.
\end{proof}
\begin{remark}
  \label{rem:comparison-smallness}
  We shown that if~$\abs{\log \bbA}_{\setBMO(B)}$ is small enough, then it controls $\abs{\bbA}_{\setBMO^2_\mu(B)}$ and therefore also
  $\abs{\bbA}_{\setBMO_\mu(B)}$. On the other hand we know from Lemma~\ref{lem:log-weakest} that $\abs{\log \bbA}_{\setBMO(B)}$ is directly controlled by $\abs{\bbA}_{\setBMO^2_\mu(B)}$.
  Based on standard John-Nirenberg estimates, it is possible to show that sufficient smallness of $\abs{\bbA}_{\setBMO_\mu(B)}$ implies that $\abs{\log \bbA}_{\setBMO^2_\mu(B)}$ can be linearly controlled by $\abs{\bbA}_{\setBMO_\mu(B)}$.
  
  So overall, once one of the three quantities $\abs{\log \bbA}_{\setBMO(B)}$,   $\abs{\bbA}_{\setBMO_\mu(B)}$ and   $\abs{\bbA}_{\setBMO^2_\mu(B)}$ is small, then they are all comparable. This allows to transfer results state in one language to the others. For example, smallness of~$\abs{\log \bbA}_{\setBMO(B)}$ implies the validity of the estimates in~\cite{CMP1} and  we obtain $\norm{\nabla u}_{L^q(\mu\,dx)} \lesssim \norm{F}_{L^q(\mu\,dx)}$ for~$q > p=2$. However, the smallness of~$\abs{\log \bbA}_{\setBMO(B)}$ depends (negative) exponentially on~$q$, see the discussion in~\cite[Remark~23]{BDGN}.
\end{remark}

\printbibliography 


\end{document}